\documentclass{amsart}

\usepackage{amssymb,amsmath,amsthm,amsfonts}
\usepackage[mathscr]{euscript}

\renewcommand{\mathcal}{\mathscr}

\renewcommand{\le}{\leqslant}

\renewcommand{\ge}{\geqslant}

\def\R {\mathbb{R}}
\def\C {\mathcal{C}}
\def\N {\mathbb{N}}

\renewcommand{\div}{\mathrm{div}}

\newcommand{\pa}{\partial}

\newcommand{\aaa}{\vartheta}

\DeclareMathOperator*{\pv}{pv}

\newtheorem{proposition}{Proposition}[section]
\newtheorem{theorem}[proposition]{Theorem}
\newtheorem*{theorem*}{Theorem}
\newtheorem{corollary}[proposition]{Corollary}
\newtheorem{lemma}[proposition]{Lemma}

\theoremstyle{definition}

\newtheorem{example}[proposition]{Example}

\numberwithin{equation}{section}

\begin{document}

\title[Stable solutions]{On stable solutions \\
of boundary reaction-diffusion equations \\
and applications to nonlocal problems \\
with Neumann data}

\author[S. Dipierro]{Serena Dipierro}
\address[Serena Dipierro]{School of Mathematics and Statistics,
University of Melbourne,
 813 Swanston St,
Parkville, VIC-3010
Melbourne, Australia,
School of Mathematics and Statistics,
University of Western Australia,
35 Stirling Highway,
Crawley, Perth,
WA-6009, Australia, 
and
Weierstra{\ss}-Institut f\"ur Angewandte
Analysis und Stochastik, Hausvogteiplatz 5/7, 10117 Berlin, Germany}
\email{serena.dipierro@ed.ac.uk}

\author[N. Soave]{Nicola Soave}
\address[Nicola Soave]{Mathematisches Institut,
Justus-Liebig-Universit\"at Giessen, 
Arndtstrasse 2, 35392 Giessen, Germany}
\email{nicola.soave@gmail.com, nicola.soave@math.uni-giessen.de}

\author[E. Valdinoci]{Enrico Valdinoci}
\address[Enrico Valdinoci]{School of Mathematics and Statistics,
University of Melbourne,
 813 Swanston St,
Parkville, VIC-3010
Melbourne, Australia,
School of Mathematics and Statistics,
University of Western Australia,
35 Stirling Highway,
Crawley, Perth,
WA-6009, Australia, 
Weierstra{\ss}-Institut f\"ur Angewandte
Analysis und Stochastik, Hausvogteiplatz 5/7, 10117 Berlin, Germany,
Dipartimento di Matematica, Universit\`a degli studi di Milano,
Via Saldini 50, 20133 Milan, Italy, and
Istituto di Matematica Applicata e Tecnologie Informatiche,
Consiglio Nazionale delle Ricerche,
Via Ferrata 1, 27100 Pavia, Italy.}
\email{enrico@math.utexas.edu}

\begin{abstract}
We study reaction-diffusion equations
in cylinders with possibly nonlinear diffusion
and possibly nonlinear Neumann boundary conditions.
We provide a geometric
Poincar\'e-type inequality and classification results for stable
solutions, and we apply them to the study of an associated nonlocal problem. We also establish a counterexample
in the corresponding framework for the fractional Laplacian.
\end{abstract}



\maketitle

\tableofcontents

\section{Introduction}

\subsection{Boundary reactions and stable solutions}

In this paper we study
reaction-diffusion equations,
i.e. mathematical models 
in which the diffusion process is in balance
with a nonlinear reaction. The diffusion
is modeled by a (possibly nonlinear) operator of elliptic type,
and the reaction may occur on the domain as well as on the
boundary, via a Neumann condition.
A typical example of reaction-diffusion model is given by the 
Peierls-Nabarro model for atom dislocations in crystals,
in which the elastic force acting on the dislocation
function is balanced through a potential acting
on the slip plane and thus producing a boundary
reaction (see e.g.~\cite{Hirth-Lothe}
or Section~2 of~\cite{DPV-CMP}
for a physical derivation of such model).

Other models which naturally produce reaction-diffusion
equations concern 
the distribution
of chemical substances, such as in the case of
the so called Fisher-KPP equation
(see~\cite{Fisher} and~\cite{Kolmogorov-Moscow-1937}).

The domain that we will consider in this paper
is a cylinder that is infinite
in one direction,
namely the Cartesian product of a smooth domain $\Omega$
and~$(0,+\infty)$. Homogeneous Neumann conditions
are prescribed along the lateral boundary~$\partial\Omega\times\{y\}$,
for any~$y>0$, and possibly nonhomogeneous and weighted Neumann data
are given on the bottom of the domain~$\Omega\times\{0\}$.
The interest for this Neumann type conditions
in cylinder is also related to the representation
of the powers of the Laplacian in the spectral sense,
see \cite{MonPelVer, Stinga, StinVolz}.

The main problem we address here is the classification
of stable solutions, i.e. solutions of the equation
which correspond to a nonnegative second
variation of the associated energy functional
(notice that, in particular, minimal solutions
fall into this category). The classification
of stable solutions of elliptic
equations with homogeneous Neumann data goes back at
least to the celebrated results in~\cite{CH},
which show that the only stable solutions
of semilinear equations in a domain with homogeneous Neumann
conditions are the constants, under suitable convexity
assumptions either on the domain or on the nonlinearity.

Our main results concern the extension
of these type of classifications for reaction-diffusion
equations on
cylindrical domains with reactive boundary conditions
(in this circumstances, as we will see,
the stable solutions are not necessarily constant,
but will depend only on the ``vertical'' variable).

Related, but rather different in spirit,
classification results for reaction-diffusions
in low-dimensional halfspaces have
been obtained in~\cite{Sola, SirVal1, SirVal2,
Cabre-Sire, Cabre-Cinti-1, savin, Cabre-Cinti-2}
(in this case, the stable solutions only depend
on one ``horizontal'' variable).

Also, we will provide a geometric 
Poincar\'e-type formula, which can be
seen as the counterpart of an inequality
obtained in~\cite{Stern}
for elliptic equations.

Since the results obtained are related
to fractional equations, we will firstly apply our main results to a Neumann boundary value problem for the spectral Neumann Laplacian. Afterwards, we provide
a counterexample that prevents classification
in a related, but different, nonlocal setting.

Now we introduce the model under consideration
in further details and give precise statements
of the results obtained.

\subsection{The mathematical setting}

The problem under investigation in this paper is the following: 
\begin{equation}\label{general problem} 
\begin{cases} \div(a(y,|\nabla 
u|) \nabla u) =g(y,u) & \text{in $\Omega \times (0,+\infty)=: \C$,} \\ 
\partial_\nu u = 0 & \text{on $\pa \Omega \times (0,+\infty)=: \pa_L 
\C$,} \\ -a(y,|\nabla u|) \pa_y u= f(u) & \text{on $\Omega \times 
\{0\}=: \pa_B \C$}. \end{cases} \end{equation} 
Here and in the rest of 
the paper, the set~$\Omega \subset \R^n$ is a bounded and sufficiently 
regular (say of class $C^{4,\alpha}$) domain.
As for the forcing terms $g$ and 
$f$, we suppose that~$g$ is continuous with respect to
the first variable and locally Lipschitz with respect to
the second variable, and that~$f\in C^{2,\alpha}(\R)$, for some~$\alpha\in(0,1)$ (we remark that $f$ locally Lipschitz would be sufficient for most of the result in the paper, with the exception of Theorems \ref{thm: s-Neumann 1} and \ref{thm: s-Neumann 2}).

The variables in the cylinder $\C$ are denoted by $x \in \Omega \subset \R^n$, and $y \in (0,+\infty)$. 
In some cases, we will use the notation $X:=(x,y) \in \mathcal{C}$.

Moreover, in the whole of the paper we will assume the following structural
conditions on the function~$a$: 
we assume that
\[
a \in C((0,+\infty) \times [0,+\infty)) \cap C^1((0,+\infty) \times (0,+\infty)),
\]
that
\begin{equation}
a(y,t)>0 \quad
{\mbox{and}}\quad
a(y,t)+ t a_t(y,t) >0 \label{ELLIPTICITY} 
\end{equation}
for any~$y>0$ and~$t\ge0$,
that
there exists $C>0$ such that
\begin{equation}
t \,|a_t(y,t)|\le C\,a(y,t) \label{ELLIPTICITY:2}
\end{equation}
for any~$y>0$ and~$t\ge0$, and that
\begin{equation}\label{ELLIPTICITY:3}
\lim_{t\to0} t\,a_t(y,t)=0
\end{equation}
for any~$y>0$. 

Here and in what follows the subscript $t$ stays for
the derivative of $a$ with respect to the second variable.
{F}rom the analytical point of view, condition~\eqref{ELLIPTICITY}
may be seen as a rather general form of ellipticity
(this will be detailed in Lemma~\ref{B-POS}).
Some examples of $a(y,t)$ that we take into account are
\begin{align*}
{\mbox{$a(y,t)=y^\aaa$, with $\aaa\in(-1,1)$,}} \\
{\mbox{$a(y,t) = y^\aaa (1+t^2)^{p/2}$,
with $\aaa\in(-1,1)$ and~$p>1$,}} \\
{\mbox{$a(y,t) = \displaystyle\frac{y^\aaa}{\sqrt{1+t^2}}$,
with $\aaa\in(-1,1)$ and
$|\nabla u| \in L^\infty(\C)$.}}
\end{align*}    
In particular, our assumptions
comprise the 
quasilinear equations of $p$-Laplace type
and mean curvature type, possibly weighted
by Muckenhoupt weights. The case $a(y,t)=y^\aaa$ naturally arises
in some extension problems for the spectral fractional Laplacian with Neumann boundary condition, see \cite{MonPelVer, Stinga}.

We now clarify the type of solutions that
we are going to consider. We always suppose that
\begin{equation}\label{hp su u-2}
\begin{split}
& u \in C(\overline{\mathcal{C}}) \cap C^2(\overline{\Omega} \times (0,+\infty) ),\\
&\nabla_x u,\;D^2_x u \in L^2(\Omega \times \{0\}), \\
&\partial_\nu u(x,y) = 0 \quad {\mbox{for all }} (x,y) \in \partial \Omega \times [0,+\infty), 
\,{\mbox{ and for all $R>0$}}\\
& a(y,|\nabla u|)\,\Big( |\nabla u|^2 + |D^2_x u|^2+|\nabla_x u_y|^2+
|D^3_x u|^2 +|D^2_x u_y|^2 \Big)\in L^1(\Omega\times (0,R)),
\end{split}  
\end{equation}
where $\nu=(\tilde \nu,0) \in \R^n \times \R$ and $\tilde \nu$ denotes the outer unit vector field on $\partial \Omega$. 

Here and in the rest of the paper, the notation~$\nabla_x$ stands for
the gradient only in the~$x$ variable (in particular, $\nabla_x u$
is an $n$-dimensional vector field).
The second condition in~\eqref{hp su u-2}
is intended in the sense that $\nabla_x u$ and $D^2_x u$ have a $L^2$ trace on $\Omega \times \{0\}$.
Notice also that the second and the third conditions in~\eqref{hp su u-2}
do not require $u$ to be of class $C^1$ near $\partial \Omega \times \{0\}$, since only the derivatives of $u$ with respect to $x$ are taken into account. We shall see that the previous assumptions are naturally satisfied when \eqref{general problem} is seen as extension of a nonlocal boundary value problem. Moreover, they can be directly checked in many concrete cases using the classical regularity theory for elliptic equation (up to the boundary), for which we refer to \cite{NIR, NIR2}.

Concerning the last equation in \eqref{general problem}, under reasonable assumptions on $a$, $g$ and $f$ it can be interpreted in the classical case as
\begin{equation}\label{last equation}
 \lim_{y\to0} f(u(x,y))+a\big(y,|\nabla u(x,y)|\big)\,
\partial_y u(x,y) = 0
\quad \text{for any~$x\in\Omega$}.
\end{equation}
In general we will not need such a regularity. On the contrary, we call solution of \eqref{general problem} any function $u$ satisfying \eqref{hp su u-2} and such that \eqref{general problem} holds in the following weak sense:
\begin{equation}\label{WEAK}
\int_{\C} a(y,|\nabla u|) \nabla u \cdot \nabla \varphi + \int_{\C}
g(y,u) \varphi = \int_{\pa_B \C} f(u) \varphi
\end{equation}
for any $\varphi\in{\mathcal{A}}$, where
\begin{equation}\label{DEF:A}
{\mathcal{A}}:= \left\{ \varphi\in W^{1,1}_{\rm loc}(\C) \left| \begin{array}{l}
{\mbox{$\varphi$ has bounded support in $y$, $a(y,|\nabla u|)\,|\nabla\varphi|^2\in L^1(\C)$}}
\\ 
 {\mbox{and }} \varphi|_{\Omega \times \{0\}} \in L^2(\Omega)
\end{array}\right.\right\}.
\end{equation}
It is clear that any classical solution is also a weak one, in the sense specified above.

Let us consider now the symmetric matrix
\begin{equation}\label{DEF:B}
\mathcal{B}(y,\eta)_{ij}:= 
a(y,|\eta|) \delta_{ij} + 
\frac{a_t(y,|\eta|)}{|\eta|}\eta_i \eta_j \qquad {\mbox{for all }} i,j =1,\dots,n+1,
\end{equation}
where~$\eta=(\eta_1,\dots,\eta_{n+1})$, and we mean that the latter term is zero
if~$\eta$ is zero.

The matrix $\mathcal{B}$ plays a role in the linearized
equation (in a sense that will be clarified in Lemma \ref{CL}). 

We write that $u$ is a \emph{stable solution} of \eqref{general problem} if it is a solution (in the sense of~\eqref{WEAK})
and if 
\begin{equation}\label{stability}
I(\varphi) :=\int_{\C} \langle \mathcal{B}(y,\nabla u)  
\nabla \varphi,\nabla \varphi\rangle + \int_{\C}g_u(y,u)\varphi^2 -
\int_{\pa_B \C} f'(u) \varphi^2 \ge 0
\end{equation}
for any $\varphi\in {\mathcal{A}}$. 

Having introduced the main definitions and notation, we are in position to present our main results,
which are: a geometric Poincar\'e-type formula,
the classification of stable solutions when $\Omega$ is convex, the
classification of bounded stable solutions in case of convex/concave boundary 
reaction~$f$, and the application of these results to nonlocal problems in $\Omega$ related to the spectral Neumann Laplacian. Finally, we also present
a counterexample for the fractional Laplacian with point-wise Neumann boundary condition.

\subsection{A Poincar\'e-type formula}

The first result that we present is
a weighted Poincar\'e-type inequality.
A weighted $L^2(\C)$-norm of any test function
will be bounded by a weighted
$L^2(\C)$-norm of its gradient.
The weights are non-negative and possess
a neat geometric interpretation. This type
of Poincar\'e-type formulas are indeed an extension
of a celebrated result obtained in~\cite{Stern}
for classical elliptic equations.
The precise statement in our framework is the following:

\begin{theorem}\label{TH:POI}
Let $u$ be a stable solution of \eqref{general problem}.
Then, for any $\psi \in C^1(\overline{\mathcal{C}})$
with bounded support in~$y$ and such that $\psi_{x_j} \in \mathcal{A}$ for any $j=1,\dots,n$, we have
\begin{equation}\label{POI:FO}
\begin{split}
& \int_\C\left[ \sum_{j=1}^{n}
\langle {\mathcal{B}}(y,\nabla u)\,\nabla u_{x_j},
\nabla u_{x_j}\rangle
-\langle \mathcal{B}(y,\nabla u)
\nabla |\nabla_x u|,\, \nabla |\nabla_x u|\rangle\right] \,\psi^2
\\ &\qquad-\int_{\partial_L\C}
a(y,|\nabla u|)\,\big(\nabla u \cdot \partial_\nu( \nabla u)\big)\,\psi^2
\le\int_{\C}\langle \mathcal{B}(y,\nabla u)
\nabla \psi,\,\nabla \psi\rangle\,|\nabla_x u|^2.
\end{split}\end{equation}
\end{theorem}

We remark that the weights in~\eqref{POI:FO}
have a simple, concrete interpretation
in terms of the level sets of the solution~$u$.
As a matter of fact, fixed~$y>0$,
if~$(x,y)\in\{u=c\}\cap\{\nabla_x u\ne 0\}$,
then the $c$-level set of~$u(\cdot, y)$ in the vicinity
of~$(x,y)$ is a smooth $(n-1)$-dimensional
manifold~${\mathcal{S}}_y$ in~$\Omega\times\{y\}$, and we can therefore
consider the tangential gradient~$\nabla_{{\mathcal{S}}_y}$ along~${\mathcal{S}}_y$
and the principal curvatures~$\kappa_1,\dots,\kappa_{n-1}$.
In this way, one can consider the norm of
the second fundamental form, i.e.
$$ {\mathcal{K}}:=\sqrt{ \sum_{i=1}^{n-1} \kappa_i^2}$$
and bound the weight on the left hand side of~\eqref{POI:FO}
in terms of these quantities.

More explicitly (see formula~(1.20) in~\cite{SirVal2}), one has that on $\mathcal{C} \cap \{\nabla_x u \neq 0\}$ it results
\begin{eqnarray*}
&& \sum_{j=1}^{n}
\langle {\mathcal{B}}(y,\nabla u)\,\nabla u_{x_j},
\nabla u_{x_j}\rangle
-\langle \mathcal{B}(y,\nabla u)
\nabla |\nabla_x u|,\, \nabla |\nabla_x u|\rangle\\
&=& a(y,|\nabla u|)\,{\mathcal{K}}_0
+\frac{a_t(y,|\nabla u|)}{|\nabla u|}\,{\mathcal{K}}_\sharp,
\end{eqnarray*}
where
\begin{eqnarray*}
&& {\mathcal{K}}_0 := \sum_{j=1}^n u_{x_j\,y}^2 - (\partial_y
|\nabla_x u|)^2
+{\mathcal{K}}^2\,|\nabla_x u|^2+
\big|\nabla_{{\mathcal{S}}_y} |\nabla_x u|\big|^2
\\ {\mbox{and }}
&& {\mathcal{K}}_\sharp := \sum_{j=1}^n (\nabla u\cdot \nabla u_{x_j})^2
-(\nabla u\cdot \nabla |\nabla_x u|)^2.
\end{eqnarray*}
Since $\nabla u_{x_j}= 0 = \nabla |\nabla_x u|$ for almost every point in $\{\nabla_x u=0\}$, 
this type of inequalities has also
a deep relevance for rigidity and symmetry results,
as pointed out by~\cite{Farina}, see also~\cite{FaScVa, SirVal1, SirVal2}.

\subsection{Classification of stable solutions in convex domains}

One of the main goal of this paper is to classify
stable solutions of~\eqref{general problem}
under suitable assumptions
either on the domain or on the nonlinearities. In this spirit, the 
first result that
we present concerns classification
in convex domains. 

\begin{theorem}\label{thm: symmetry convex 1}
Let $n\ge2$, $\Omega \subset \R^n$ be convex,
with strictly positive principal curvatures along~$\partial\Omega$.

Let $u$ be a stable solution of \eqref{general problem} satisfying the
energy bound, for any $R\ge1$,
\begin{equation}\label{energy bound}
\int_{\Omega \times (0,R)} a(y,|\nabla u|) 
\,|\nabla_x u|^2 \le C R^2,
\end{equation}
for some constant $C>0$ independent of $R$. 

Then $u$ depends only on $y$.\end{theorem}

We remark that
assumption \eqref{energy bound} is satisfied, 
for instance, if $a$ grows in~$y$
at infinity as~$y^\aaa$ (with~$\aaa\in(-1,1)$)
and~$|\nabla_x u|$ has growth bounded by~$y^{\frac{1-\aaa}{2}}$;
in particular, unbounded solutions may be also taken
into account. 

We also stress that, in general (and differently from the setting
in~\cite{CH}),
it is not possible to deduce, in the setting
of Theorem~\ref{thm: symmetry convex 1}, that the solution~$u$
is constant (as a counterexample, one may consider the
case in which~$u:=y$, $a:=1$, $f:=-1$, $g:=0$; clearly, $u$ is stable being a harmonic function).\medskip

We think it is an interesting problem to detect the maximal generality
under which this type of results holds true. To this aim, we observe that
it is possible to obtain the same result removing the assumption
of strict positivity of the curvature of the domain, but adding an integrability condition,
as stated in the following result:

\begin{theorem}\label{P-CA}
Let $\Omega\subset\R^n$ be convex, and let $u$ be a stable
solution of \eqref{general problem}, satisfying the integrability assumption 
\begin{equation}\label{integrabilty assumption}
a(y,|\nabla u|)\, 
\Big( |\nabla u|^2 + |D^2_x u|^2 +|\nabla_x u_y|^2\Big) + \big|
g_u(y,u)\big|\, |\nabla_x u|^2 
\in L^1(\C).
\end{equation}
Then $u$ depends only on $y$.
\end{theorem}

We observe that the stability condition 
in Theorem~\ref{P-CA} cannot be dropped.
As an example, one can consider the domain~$\Omega:=(0,2\pi)\subset\R$
and the function~$u:\Omega\times[0,+\infty)$ given by~$u(x,y):=
e^{-y}\,\cos x$.
Then, we see that~$u$ is a solution of~\eqref{general problem}
with~$a:=1$, $g:=0$ and~$f(u):=u$; also, it satisfies~\eqref{integrabilty assumption},
but it is not a function only of~$y$ (comparing with Theorem~\ref{P-CA},
we have that~$u$ is not stable, and the stability assumption cannot be removed).
\medskip

Moreover, we observe that
assumption \eqref{integrabilty assumption} in Theorem~\ref{P-CA} cannot be dropped.
As an example, one can consider the domain~$\Omega:=(0,2\pi)\subset\R$
and the function~$u:\Omega\times[0,+\infty)$ given by~$u(x,y):=
e^{y}\,\cos x$.
Then, we see that~$u$ is a solution of~\eqref{general problem}
with~$a:=1$, $g:=0$ and~$f(u):=-u$; therefore, $f'=-1$ and so $u$ is stable
(but it does not
satisfy~\eqref{integrabilty assumption}, and it does not depend only on $y$,
showing that Theorem~\ref{P-CA} is optimal in this sense).
\medskip

It is also worth to observe that \eqref{integrabilty assumption} can be 
considerably weakened when $a(y,t)$ is independent of $t$ and $g \equiv 
0$. This case is particularly interesting, as we will discuss
in the forthcoming Section~\ref{KH:9veto}.
%
%
\medskip 

In the case of classical elliptic equations, a classification
of stable solutions in convex domains with homogeneous Neumann boundary data
was given in~\cite{CH}.
Indeed, our Theorems~\ref{thm: symmetry convex 1} and \ref{P-CA} may be seen as the extension
of Theorem~2 of~\cite{CH} to the case of boundary reaction-diffusion
equations.

In many concrete cases,
once one knows that the solution only depends on~$y$
(as given for instance by Theorem~\ref{P-CA}),
then~\eqref{general problem} simplifies and can be often explicitly
integrated. For instance, if~$u=u(y)$ and~$a=a(y)$ only depend on~$y$,
and~$g$ vanishes identically,
then~\eqref{general problem} reduces to an ordinary
differential equation which provides the family of solutions
\begin{equation} \label{O76:98}
u(y)=c-f(c)\,\int_0^y \frac{d\zeta}{a(\zeta)},\end{equation}
for $c\in\R$.
We also remark that in the model case
in which~$a(y)=y^\aaa$, with~$\aaa\in(-1,1)$,
the functions~$u$ 
of the form~\eqref{O76:98} that satisfy~$au_y^2\in L^1(\C)$
are the constants
(see also Lemma~4.10 in~\cite{FaScVa}
for classification results of
ordinary
differential equations).

Moreover, we stress that, in general,
stable solutions of \eqref{general problem}
are \emph{not necessarily constant}. 
As an example, one can consider~$u(x,y):=e^{-y}$,
$a(y,t):= e^y$, $f:=1$ and~$g:=0$
(notice that in this case
\eqref{general problem}, \eqref{hp su u-2}
and~\eqref{integrabilty assumption}
are all satisfied; moreover, since $g \equiv 0$ and $f' \le 0$, 
and recalling Lemma~\ref{B-POS},
then we see that~$u$
is stable, in the sense of~\eqref{stability}). In this sense,
Theorems~\ref{thm: symmetry convex 1} and~\ref{P-CA} are optimal.
\medskip

Though the statements of
Theorems~\ref{thm: symmetry convex 1} and~\ref{P-CA}
are somehow similar, we prove them by different methods.
Indeed, the proof of Theorem~\ref{thm: symmetry convex 1} 
relies on the 
Poincar\'e-type geometric inequality
stated in Theorem~\ref{TH:POI} (by the choice of
an appropriate test function), while the proof of 
Theorem~\ref{P-CA} is based on the relation between maximum
principles and stability conditions in view of a suitable
spectral analysis.

\subsection{Classification of bounded stable solutions for convex/concave
nonlinearities}

Now we address the problem of classifying stable
solutions if the nonlinearity $f$ is either convex or concave. To this aim, we shall make the assumption that $g \equiv 0$.
The precise result obtained is the following:

\begin{theorem}\label{TH:CC:1}
Assume that
\begin{equation}\label{LO}
{\mbox{$a_t(y,t)\le0$ and $g(y,t) = 0$ for any~$t$, $y>0$.}}\end{equation}
Let $u$ be a bounded and stable solution
of~\eqref{general problem}, such that
\begin{equation}\label{NUOVE}
\begin{split}
a(y,|\nabla u|) |\nabla u|^2 \in L^1(\C) ,\\ 
\lim_{R\to+\infty}\frac{1}{R^2}
\int_{\Omega\times(R,2R)} a(y,|\nabla u|) =0 \\
{\mbox{and }}\qquad
a(y,|\nabla u|)\, \pa_y u \in C(\overline\Omega \times [0,+\infty))
.\end{split}
\end{equation}
If either $f$ is strictly convex, or $f$ is strictly concave, then~$u$ is constant in $\C$.
\end{theorem}

We remark that Theorem~\ref{TH:CC:1}
is proved here under the additional
assumption in~\eqref{LO}, stating that~$a$ is nonincreasing
with respect to the variable~$t$.
This assumption is of course satisfied in all the cases in
which~$a$ is independent of~$t$, 
that is, if one is considering
semilinear reaction-diffusion equation.
Nevertheless, we remark that condition~\eqref{LO}
is satisfied also in the case of important quasilinear
reaction-diffusion equations, such as the one
driven by mean curvature-type operators, in which
$$ a(y,t)=\frac{y^\aaa}{\sqrt{1+t^2}},$$
with~$\aaa\in(-1,1)$.
We think that it is an interesting open problem to
decide for which type of quasilinear reaction-diffusion equations
similar statements hold true.

In the case of elliptic equations with inner reaction,
the classification of stable solutions
with Neumann data under suitable convexity or concavity assumptions
on the nonlinear term was obtained in~\cite{CH}.
In this sense, our Theorems~\ref{TH:CC:1}
is the extension of Theorem~3
of~\cite{CH} to the
boundary reaction-diffusion equation in~\eqref{general problem}.

\subsection{Application to nonlocal Neumann problems}\label{KH:9veto}

Now we discuss the classification of stable
solutions in a problem driven by the square root of the Laplacian
in the spectral sense.
For this scope,
let $\{\varphi_k: k \in \N \cup \{0\}\}$ and $\{\lambda_k: k \in \N \cup \{0\}\}$ be the eigenfunctions and the
eigenvalues of $-\Delta$ in $\Omega$ with homogeneous 
Neumann conditions on $\pa \Omega$. We normalize the sequence of
eigenfunctions in such a way that they form an orthonormal basis of~$L^2(\Omega)$.

The Neumann Laplacian $-\Delta_N$ is the operator acting on an $L^2(\Omega)$-function 
$$ w(x)= \sum_{k=0}^\infty w_k \,\varphi_k(x),$$
where
$$ w_k:=\int_\Omega w(x)\,\varphi_k(x)\,dx,$$ 
as
\[
-\Delta_N w(x): = \sum_{k=0}^\infty \lambda_k \,w_k\, \varphi_k(x).
\]
Then, for $s \in (0,1)$, the $s$-Neumann Laplacian is given by 
\begin{equation}\label{HG:sn}
(-\Delta_N)^s w(x):= \sum_{k=0}^\infty \lambda_k^s\, w_k\, \varphi_k(x).
\end{equation}
With the language of the semigroups introduced in \cite{Stinga}, it is possible to show that $(-\Delta_N)^s$ is a nonlocal operator. 

From now on we focus on the case $s=1/2$
and, given~$f\in C^{2,\alpha}(\R)$, we consider the semilinear equation
\begin{equation}\label{s-Neumann}
\begin{cases}
(-\Delta_N)^{1/2} v = f(v) & \text{in $\Omega$} \\
\pa_\nu v=0 & \text{on $\partial \Omega$}.
\end{cases}
\end{equation}
The problem can be considered in weak sense, namely 
we consider the space
\[
H^{1/2}(\Omega):= \left\{ w =\sum_{k=0}^{+\infty} w_k\varphi_k
\in L^2(\Omega) {\mbox{ s.t. }}\sum_{k=0}^{+\infty} \lambda_k^{1/2} |w_k|^2  <+\infty\right\}.
\]
Then we say that~$v$ is a solution 
of~\eqref{s-Neumann} if 
$$v=\sum_{k=0}^{+\infty} v_k\varphi_k \in H^{1/2}(\Omega)$$
and
$$ \sum_{k=0}^{+\infty} \lambda_k^{1/2}\,v_k \,\zeta_k
=
\int_{\Omega} f(v(x))\,\zeta(x)\,dx\qquad
\;{\mbox{ for any }}
\zeta=\sum_{k=0}^{+\infty} \zeta_k\varphi_k \in H^{1/2}(\Omega).$$
We observe that the latter integral makes sense under some assumption on $f$ or on $v$. Since $f$ is continuous and $v$ will always be bounded in the sequel, it is well defined.
Also, thanks to the results in \cite{Stinga, StinVolz} (see also \cite{MonPelVer}), the previous nonlocal problem is related to the following local
one, with boundary reaction:
\begin{equation}\label{s-Neumann extended}
\begin{cases}
\Delta u = 0 & \text{in $\C$} \\
\pa_\nu u=0 & \text{on $\pa_L \C$} \\
-\pa_y u = f(u) & \text{on $\pa_B \C$}.
\end{cases}
\end{equation}
More precisely,
let us define $\mathcal{H}(\C)$ as the completion of $H^1(\C)$ with respect to the scalar product
\[
(u_1,u_2)_{ {\mathcal{H}}(\C) }:= \int_{\C} \nabla u_1 \cdot \nabla u_2 + \int_{\Omega} u_1 u_2,
\]
where $u_i|_{\Omega}$ has to be understood in the sense of traces (notice that this is possible, see Section 2 in \cite{StinVolz}). It results that~$
\mathcal{H}(\C) \supset H^1(\C)$ (notice in particular that constant functions are in $\mathcal{H}(\C)$ but not in $H^1(\C)$; for this reason, $
\mathcal{H}(\C)$ is a more
suitable
space than $H^1(\C)$ to set \eqref{s-Neumann extended} in weak sense). Then a weak solution $v$ to \eqref{s-Neumann} can be defined as the trace over $\Omega$ of a function $u \in \mathcal{H}(\C)$ such that
\begin{equation}\label{po4rBH}
\int_{\C} \nabla u \cdot \nabla \varphi - \int_{\Omega\times\{0\}} f(u) \varphi = 0 \qquad {\mbox{for all }} \varphi \in \mathcal{H}(\C),
\end{equation}
see again \cite{StinVolz}. Notice that this setting
falls exactly under the general setting considered in \eqref{general problem},
with~$a\equiv1$ and~$g\equiv0$
(compare~\eqref{po4rBH} with~\eqref{WEAK}).
\medskip

In this framework, we say that a
solution~$v \in H^{1/2}(\Omega)$ to \eqref{s-Neumann}
is \emph{stable} if its extension $u \in \mathcal{H}(\C)$ 
in~\eqref{s-Neumann extended}
is stable
according to~\eqref{stability} (with ${\mathcal{B}}$ the identity
matrix and $g \equiv 0$), i.e.
\begin{equation}\label{stability:EXT:TY}
\int_{\C} 
|\nabla \varphi|^2
-\int_{\Omega\times\{0\}} f'(u) \varphi^2 \ge 0
\end{equation}
for any $\varphi\in {\mathcal{A}}$.
\medskip

With this definition,
we can prove the following classification theorems
for stable solutions to \eqref{s-Neumann}.

\begin{theorem}\label{thm: s-Neumann 1}
Let $\Omega\subset\R^n$ be convex,
and let $v \in H^{1/2}(\Omega) \cap L^\infty(\Omega)$ be a stable solution to \eqref{s-Neumann}. Then $v$ is constant.
\end{theorem}

Theorem \ref{thm: s-Neumann 1}
establishes that equation~\eqref{s-Neumann} 
in convex domains does not admit noncostant stable solutions.
The same conclusion holds, if, instead of the convexity of $\Omega$, 
we assume the convexity (or the concavity) of the nonlinearity $f$,
according to the following result.

\begin{theorem}\label{thm: s-Neumann 2}
Let $f$ be either strictly convex, or strictly concave, and let $v \in H^{1/2}(\Omega) \cap L^\infty(\Omega)$ be a stable solution to \eqref{s-Neumann}. Then $v$ is constant.
\end{theorem}

The previous results can be considered as the counterpart of those in \cite{CH} for \eqref{s-Neumann}. Clearly, a natural question consist in finding easy and natural assumptions on $f$ or on $v$ allowing to show that $v$ is stable in the sense of~\eqref{stability:EXT:TY}. A very simple condition consists in $f' \le 0$.

We also point out that, in our framework, Theorems~\ref{thm: s-Neumann 1}
and~\ref{thm: s-Neumann 2} will be obtained by using
Theorems~\ref{P-CA}
and~\ref{TH:CC:1}, respectively.\medskip

As a further remark, we observe that
if $v \in H^{1/2}(\Omega) \cap L^\infty(\Omega)$, 
then, by \cite[Theorem 3.5]{StinVolz}, we have that~$
v \in \mathcal{C}^1(\overline{\Omega})$. Therefore the boundary condition $\pa_\nu v=0$ on $\pa \Omega$ can be understood in the classical sense.\medskip

We also mention that our focus
on the case $s=1/2$ is due to the fact that
we recalled and used some results contained in \cite{StinVolz}. 
Once that similar results are established
for the case $s \neq 1/2$ (this is 
announced in \cite{StinVolz}), our results would also hold
for the general case~$s\in (0,1)$.
Indeed, for $s \in(0,1)$, the extension problem associated to the 
$s$-Neumann Laplacian will be of type \eqref{general problem} with 
$a(y,t)=y^{\aaa}$, with~$\aaa \in (-1,1)$, and $g \equiv 0$. 

\subsection{A counterexample in a different nonlocal setting}

In Section~\ref{KH:9veto},
we have considered classification results for
stable solutions of spectral versions of
fractional Laplacians (in the sense given by~\eqref{HG:sn}).

In the literature, other nonlocal elliptic operators of fractional
type have been widely studied.
Of particular interest is the integral version of
the fractional Laplacian, defined (up to normalizing constants),
for any~$s\in(0,1)$, as
\begin{equation}\label{HG:sn:2}
(-\Delta)^s v (x):= \pv \int_{\R^n} 
\frac{v(x)-v(y)}{|x-y|^{n+2s}}\,dy. \end{equation}
As usual, $\textrm{pv}$ stays for 
the principal value. We stress that the operators in~\eqref{HG:sn}
and~\eqref{HG:sn:2} are indeed different (see e.g.~\cite{different}).

In this setting, a natural fractional 
normal derivative at the boundary
(see e.g. \cite{ROS}) is given by
\[
(\pa_\nu)^s v(x): = \lim_{t \to 0^+} \frac{v(x+t \tilde\nu(x))-v(x)}{t^s},
\]
where $\tilde \nu(x)$ denotes the outer unit normal to $\pa \Omega$ at $x \in \pa \Omega$.

With this, one may wonder
whether ``nice'' and ``stable'' solutions to the equation
\begin{equation}\label{HG:sn:2:PRE} 
\begin{cases}
(-\Delta)^s v=f(v) & \text{in $\Omega$} \\
(\pa_\nu)^s v=0 & \text{on $\partial \Omega$}
\end{cases}
\end{equation}
in convex domains or with convex nonlinearities are necessarily constant,
or, at least, if they enjoy some rigid geometric properties.\medskip

While a suitable notion of stability should be introduced
in this setting, the further assumption that~$f\equiv0$
would imply stability in any reasonable definition, thus
the basic question boils down to determine any
rigidity properties of solutions of
\begin{equation}\label{GaTf7y}
\begin{cases}
(-\Delta)^s v=0 & \text{in $\Omega$} \\
(\pa_\nu)^s v=0 & \text{on $\partial \Omega$},
\end{cases}
\end{equation}
possibly in convex domains.\medskip

Quite surprisingly, we now show that no classification
(and even no rigidity) results hold true for equation~\eqref{GaTf7y}.
This phenomenon shows that
the ``right'' choice of fractional operator, endowed with
the appropriate boundary conditions, plays
a crucial role in nonlocal problems. \medskip

In concrete, the result that we show is the following:

\begin{example}\label{EXAMPLE}
Let~$s\in(0,1)$, $h\in C^2(\R)$ and~$\epsilon\in(0,1)$. Then, there exist~$\delta_1,\delta_2\in
\left[0,\frac{\epsilon}{2}\right]$
and~$v\in C^2( (-1-\epsilon,1+\epsilon))\cap C(\R)$ such that
\begin{equation*}
\begin{split}
& \|v-h\|_{C^2 ((-1,1))}\le\epsilon,\\
& (-\partial^2_x)^s v(x):= 
\pv\int_\R \frac{v(x)-v(y)}{|x-y|^{1+2s}}\,dy=0
{\mbox{ for any }} x\in(-1-\delta_1,1+\delta_2),\\
& (\partial_\nu)^s v(x)= \lim_{\substack{ y\in (-1-\delta_1,1+\delta_2)\\ y\to x}}
\frac{v(x)-v(y)}{|x-y|^s}=0
{\mbox{ for any }} x\in\{-1-\delta_1,1+\delta_2\},\\
& v'(-1-\delta_1)=v'(1+\delta_2)=0,\\
& {\mbox{$v$ has bounded support.}}
\end{split}
\end{equation*}
\end{example}

We remark that the operator~$(-\partial^2_x)^s$
is simply~$(-\Delta)^s$, as defined in~\eqref{HG:sn:2},
when the domain is one-dimensional.
Also,
the ``fractional'' boundary derivative~$(\partial_\nu)^s$
has nice regularity properties
and natural applications in Pohozaev-type identities
and in rigidity results for overdetermined problems
(see e.g.~\cite{Grubb, ROS, ROS-2, FALL, Nicola});  nevertheless
it cannot characterize solutions~$v$ of the
fractional equation~\eqref{HG:sn:2:PRE} in convex domains,
which, as stated in Example~\ref{EXAMPLE}, at least for~$f\equiv0$,
can have essentially the same local qualitative properties
of any prescribed function~$h$. 

\subsection{Organization of the paper}

The rest of the paper is organized as follows.
In Section~\ref{sec: preliminaries},
we collect some preliminary computations
that are needed in the proofs of the main results. In Section~\ref{SEC:POINC-2}, we prove
the Poincar\'e-type geometric inequality
stated in Theorem~\ref{TH:POI}. 
The classification of solutions to \eqref{general problem} when $\Omega$ 
is convex, together with the proofs
of Theorems~\ref{thm: symmetry convex 1}
and~\ref{P-CA},
is contained in Sections~\ref{KJ:8a9} and~\ref{HG:SEC}. The 
proof of Theorems~\ref{TH:CC:1}, with the
classification of stable solutions in case of convex/concave
nonlinearities, is contained in Section~\ref{LK:COCO-0}. Section \ref{sec: s-Neumann} is devoted to the study of classification results involving the spectral $s$-Neumann Laplacian $(-\Delta_N)^s$. Section~\ref{HJ:AO} contains the discussion related to
Example~\ref{EXAMPLE}. 

\section{Toolbox}\label{sec: preliminaries}

In this section we collect several intermediate statements which will be used in the proof of our main results.

\subsection{Some inequalities coming from the Neumann condition}

Next result deals with the geometric analysis
related to functions satisfying a Neumann condition.

\begin{lemma}\label{abbiamo}
Let $\Omega\subset\R^n$ be an open set with boundary of class~$C^2$.
Let $u\in C^2(\overline\Omega\times(0,+\infty))$,
with~$\partial_\nu u=0$ on~$\partial\Omega \times(0,+\infty)$.

Assume that~$\bar x=(\bar x',\bar x_n) \in \partial \Omega$ and that
in a neighborhood of~$\bar x$ the domain~$\Omega$ can be
written in normal coordinates as the epigraph of a function~$\gamma\in
C^2(\R^{n-1})$, i.e.
$$ \Omega\cap B_r (\bar x)=
\{ x=(x',x_n)\in B_r (\bar x)
{\mbox{ s.t. }}x_n >\gamma(x')\},$$
for some~$r>0$, with~$\gamma(\bar x')=\bar x_n$ and~$\nabla\gamma(\bar x')=0$.
Then,
for any~$y>0$,
\begin{equation}\label{key-eq-EUCLIDEAN}
\begin{split}
&\nabla u(\bar x,y) \cdot \partial_\nu \left(\nabla u(\bar x,y)\right)
=\nabla_x u(\bar x,y) \cdot \partial_\nu \left(\nabla_x u(\bar x,y)\right)
\\ &\qquad= -\sum_{i,j=1}^{n-1}\gamma_{x_i x_j}
(\bar x')\,u_{x_i}(\bar x,y)\,u_{x_j}(\bar x,y)
.\end{split}\end{equation}
In particular, if~$\Omega$ is convex, then
\begin{equation}\label{key-eq-EUCLIDEAN-2}
\nabla u(\bar x,y) \cdot \partial_\nu \left(\nabla u(\bar x,y)\right)
=\nabla_x u(\bar x,y) \cdot \partial_\nu \left(\nabla_x u(\bar x,y)\right)
\le0.\end{equation}
\end{lemma}

\begin{proof} Up to a translation, we can assume that~$\bar x=0$. 
Thus, in the vicinity of the origin we can 
write the unit normal vector as 
\begin{equation*}
\nu(x,y) =\frac{1}{\sqrt{|\nabla \gamma(x')|^2 +1}} \left(\nabla 
\gamma(x'), -1,0\right) .\end{equation*}
Therefore, the condition $\partial_\nu u(x,y)=0$, for $x\in\partial\Omega\cap B_r$,
reads as 
\begin{eqnarray*}
\sum_{i=1}^{n-1}u_{x_i}\big(x',\gamma(x'),y\big)\, 
\gamma_{x_i}(x')  - u_{x_n}\big(x',\gamma(x'),y\big) =0.
\end{eqnarray*}
So, taking the derivative with respect to $x_j$ with $j=1,\dots,n-1$, we obtain
\begin{multline*}
\sum_{i=1}^{n-1}u_{x_i x_j}\big(x',\gamma(x'),y\big) 
\gamma_{x_i}(x') \\
+  \sum_{i=1}^{n-1}u_{x_i x_n}\big(x',\gamma(x'),y\big)\,
\gamma_{x_j}(x')\,\gamma_{x_i}(x')  +  
 \sum_{i=1}^{n-1}u_{x_i}\big(x',\gamma(x')\big)\,
\gamma_{x_i x_j}(x')  \\
 - u_{x_j x_n}\big(x',\gamma(x')\big) - u_{x_n x_n}\big(x',\gamma(x'),y\big)\, 
\gamma_{x_j}(x')=0.
\end{multline*}
Hence, recalling that $\nabla\gamma(0') = 0$, we infer that 
$$ \sum_{i=1}^{n-1}u_{x_i}(0,y)\gamma_{x_i x_j}(0,y) - u_{x_j x_n}(0,y)=0,$$
which proves
that,
for any~$y>0$ and any~$j=1,\dots,n-1$,
\begin{equation}\label{Ki78}
u_{x_j x_n}(\bar x,y)=
\sum_{i=1}^{n-1}u_{x_i}(\bar x,y)\gamma_{x_i x_j}(\bar x,y)
.\end{equation}
Now we observe that
\begin{equation}\label{h67889}
{\mbox{
$\nu(0,y)=-e_n$ and so $
u_{x_n}(0,y)=-\partial_\nu u(0,y)=0$.}} \end{equation} 
By differentiating this identity in~$y$, we deduce that
\begin{equation}\label{h67889-2}
u_{x_n y}(0,y)=0.\end{equation}
Moreover, 
\begin{equation}\label{7uhYY}
u_y(0,y)\cdot \partial_\nu u_y (0,y)=
-u_y(0,y)\,u_{x_n y}(0,y)=0. \end{equation}
Therefore, using again~\eqref{h67889}, we see that
\begin{equation}\begin{split} \label{lsaggrthtg} \nabla u(0,y) &\cdot 
\partial_\nu \left(\nabla u(0,y)\right)= -\nabla_x u(0,y)\cdot 
\partial_{x_n} \left(\nabla_x u(0,y)\right) \\& = -\sum_{i=1}^n 
u_{x_i}(0,y) u_{x_i x_n}(0,y) = -\sum_{i=1}^{n-1} u_{x_i}(0,y) 
u_{x_ix_n}(0,y).\end{split}\end{equation}
Now we plug~\eqref{Ki78} into
\eqref{lsaggrthtg}, and we deduce that
$$ \nabla u(0,y) \cdot \partial_\nu \left(\nabla u(0,y)\right) 
= -\sum_{i,j=1}^{n-1}\gamma_{x_i x_j}(0')\,u_{x_i}(0,y)\,u_{x_j}(0,y).$$
{F}rom this and~\eqref{7uhYY}, we obtain~\eqref{key-eq-EUCLIDEAN}. Formula~\eqref{key-eq-EUCLIDEAN-2}
follows from~\eqref{key-eq-EUCLIDEAN}
and the convexity of~$\Omega$ (which boils down to the
convexity of~$\gamma$).
\end{proof}

For completeness, we also recall the following result:

\begin{lemma}\label{con:PROI}
Let~$n\ge2$.
If~$\Omega$ is convex, then~$\partial\Omega$ is pathwise
connected.
\end{lemma}

\begin{proof} Fix~$Z\in \Omega$.
Given two points~$A$, $B\in\partial\Omega$ we will
construct a path joining $A$ to~$B$ and lying on~$\partial\Omega$.
For this, we consider the segment~$S$
that joins $A$ and~$B$. Notice that~$S\subseteq\overline\Omega$,
by convexity. Given any~$X\in S$,
we can consider the halfline~$r(X)$ that emanates from~$Z$
and passes through~$X$. We remark that~$r(X)$
must intersect~$\partial\Omega$, since~$\Omega$ is bounded.
This intersection point must be unique:
indeed, if there are two points~$Y_1$, $Y_2\in \partial\Omega
\cap r(X)$, since~$B_\rho(Z)\subset\Omega$ for some~$\rho>0$,
we have that the convex hull of~$Y_1$, $Y_2$ and~$B_\rho(Z)$
lies in~$\overline\Omega$, and this contradicts that both~$Y_1$
and~$Y_2$ are boundary point.

Therefore, for any~$X\in S$, we can define a continuous
function~$\pi:S\to\partial\Omega$, as the intersection of~$r(X)$
with~$\partial\Omega$. Then, the image of~$S$ via~$\pi$ provides
the desired path lying on~$\partial\Omega$ and joining $A$ to~$B$.
\end{proof}

\subsection{Positive definiteness of~${\mathcal{B}}$}

\begin{lemma}\label{B-POS}
For any~$y>0$ and~$\eta\in\R^{n+1}\setminus\{0\}$,
the matrix~$\mathcal{B}(y,\eta)$
is positive definite.

More precisely, the matrix~$\mathcal{B}(y,\eta)$
has eigenvalues $a(y,|\eta|)+|\eta|\,a_t(y,|\eta|)$
(with multiplicity~$1$) and~$a(y,|\eta|)$ (with multiplicity~$n$).
\end{lemma}

\begin{proof} The second statement implies the
first one, thanks to~\eqref{ELLIPTICITY}. So we focus on the
proof of the second statement. For this, we fix
an orthonormal basis of~$\R^{n+1}$, say~$\{E_1,\dots,E_{n+1}\}$, such that~$
E_1:= \eta/|\eta|$. We will use this basis to diagonalize the
matrix~$\mathcal{B}(y,\eta)$. Indeed, 
for any~$k=2,\dots,n+1$, we have that~$E_1\cdot E_k=0$. Thus,
for any~$i=1,\dots,n+1$
\begin{align*}
\big( \mathcal{B}(y,\eta)\,E_1\big)_i
&=\sum_{j=1}^{n+1} \mathcal{B}(y,\eta)_{ij}\,\frac{\eta_j}{|\eta|} = a(y,|\eta|) \frac{\eta_i}{|\eta|}+
\sum_{j=1}^{n+1} 
\frac{a_t(y,|\eta|)}{|\eta|^2}\eta_i \eta_j^2 \\
& = \big( a(y,|\eta|)+|\eta|\,a_t(y,|\eta|)\big)\,(E_1)_i,
\end{align*}
while for any~$k=2,\dots,n+1$
\begin{align*}
\big( \mathcal{B}(y,\eta)\,E_k\big)_i
&= a(y,|\eta|) (E_k)_i+
\sum_{j=1}^{n+1}
\frac{a_t(y,|\eta|)}{|\eta|}\eta_i \eta_j (E_k)_j\\
&=
a(y,|\eta|) (E_k)_i+
{a_t(y,|\eta|)}\,\eta_i E_1\cdot E_k =
a(y,|\eta|) (E_k)_i. \qedhere
\end{align*}
\end{proof}

\subsection{Some results available in the literature}

Here we recall some known auxiliary statements, which will
be needed in the proof of our main results
(these statements have been included for the facility of the reader,
to make the paper more self-contained).

The following is a variant of Lemma 10 in \cite{SirVal1}. The proof can be easily obtained modifying the argument therein, and thus is omitted.

\begin{lemma}\label{herein}
Let $R>0$ and $h: \Omega \times (0,R) \to \R$ be a nonnegative measurable function. For any $\rho \in (0,R)$, let
\[
\eta(\rho):= \int_{\Omega \times (0,R)} h.
\]
Then
\[
\int_{\Omega \times (\sqrt{R},R)} \frac{h(X)}{y^2}\,dX \le 2 \int_{\sqrt{R}}^R t^{-3} \eta(t) \,dt + \frac{\eta(R)}{R^2}.
\]
\end{lemma}

\begin{lemma}[Lemma 4.2 in \cite{SirVal2}]
In $\C\cap \{ \nabla_x u\ne0\}$, it results
\begin{equation}\label{OIhh}
 \sum_{j=1}^n \langle \mathcal{B}(y,\nabla u) \,\nabla u_{x_j},\,
\nabla u_{x_j} \rangle -  \langle \mathcal{B}(y,\nabla u)\,
\nabla |\nabla_x u|, \,\nabla |\nabla_x u| \rangle \ge 0.
\end{equation}
\end{lemma}

\begin{corollary}\label{Ut:OI}
Let $x_o\in\Omega$ and~$y_o>0$. Assume that
\begin{equation}\label{7uJAJOAOAAAA}
\nabla_x u(x_o,y_o)
\neq 0\end{equation}
and that
\[
\sum_{j=1}^n \langle \mathcal{B}(y,\nabla u)
\nabla u_{x_j}, \nabla u_{x_j} \rangle - 
\langle \mathcal{B}(y,\nabla u) \nabla |\nabla_x u|, \,
\nabla |\nabla_x u| \rangle = 0
\]
almost everywhere in a neighborhood of~$\Omega\times\{y_o\}$. 

Then, each connected component of the
level set
of the function~$\Omega\ni x\mapsto u(x,y_o)$
must be
an $(n-2)$-dimensional hyperplane intersected~$\Omega$.
\end{corollary}

\begin{proof} Let~$L(x_o)$ be the connected component of the level set
$$ \{ x\in\Omega {\mbox{ s.t. }}u(x,y_o)=u(x_o,y_o)\}$$
which contains~$x_o$.
We also take~$U_o$ to be the largest possible
neighborhood of~$x_o$ in~$\Omega$
such that~$\nabla_x u\ne0$ in~$U_o$.

By Corollary 4.3 in \cite{SirVal2} (see in particular the first identity
in~(4.5) there),
we know that all the level sets of the function~$U_o\ni x\mapsto u(x,y_o)$
have vanishing principal curvatures in~$U_o$.

Accordingly, the normal of each of these level
sets is constant, and so each level set must be
contained in a hyperplane (see e.g.
Lemma 2.10 in \cite{FaScVa}). This says that
\begin{equation}\label{YS:001}
L(x_o)\cap U_o\subseteq \{ \omega_o \cdot (x-x_o)=0\},\end{equation}
for a suitable~$\omega_o\in S^{n-2}$.
Notice that~$\omega_o$ here above depends on~$(x_o,y_o)$.

Now we claim that
\begin{equation}\label{YS:002}
L(x_o)=\{ \omega_o \cdot (x-x_o)=0\}\cap\Omega.\end{equation}
Indeed, by
Corollary 4.3 in \cite{SirVal2} (see in particular the second identity
in~(4.5) there), we know that the tangential gradient
of~$|\nabla_x u|$ along~$L(x_o)\cap U_o$ is zero, and so~$|\nabla_x u|$
is constant along~$L(x_o)\cap U_o$. That is, by~\eqref{7uJAJOAOAAAA},
\begin{equation}\label{OJAJOPAFGHJKA}{\mbox{$
|\nabla_x u(x,y_o)|=|\nabla_x u(x_o,y_o)|=:c_o >0$
for any~$x\in L(x_o)\cap U_o$}}.\end{equation}
Hence, by continuity, we have that~$|\nabla_x u(x,y_o)|=c_o>0$
for any~$x\in L(x_o)\cap (\partial U_o)$. This (and the maximality
assumption on~$U_o$) imply that
\begin{equation}\label{YS:003}
{\mbox{$L(x_o)\cap (\partial U_o)$ is empty inside~$\Omega$.}}\end{equation}
This implies that
\begin{equation}\label{YS:004}
L(x_o)\subseteq\{ \omega_o \cdot (x-x_o)=0\}.\end{equation}
Indeed, if not, there would exist~$p\in L(x_o)\setminus\{ \omega_o \cdot (x-x_o)=0\}$.
By~\eqref{YS:001}, we know that~$p\in\Omega\setminus U_o$.
By considering a path joining~$p$ to~$x_o$ inside~$L(x_o)$,
we would then find a point in~$L(x_o)\cap(\partial U_o)$. This
would contradict~\eqref{YS:003}, and thus~\eqref{YS:004}
is proved.

Now, since~$L(x_o)$ is a subset of~$\Omega$, we can write~\eqref{YS:004} as
\begin{equation}\label{YS:005}
L(x_o)\subseteq\{ \omega_o \cdot (x-x_o)=0\}\cap\Omega.\end{equation}
Hence, to complete the proof
of~\eqref{YS:002}, we need to show the converse inclusion.
To this aim, one uses~\eqref{OJAJOPAFGHJKA} and~\eqref{YS:003}, to write~$L(x_o)$
as a smooth manifold without boundary in~$\Omega$. Such manifold, in view of~\eqref{YS:005},
must coincide with~$\{ \omega_o \cdot (x-x_o)=0\}$,
thus proving~\eqref{YS:002}, which is the desired result.
\end{proof}

\subsection{The linearized equation}

Now we consider the so-called linearized equation,
that is the equation satisfied by the derivatives
of the solution in the variables~$x$
(this equation is clearly related with the
stability condition in~\eqref{stability}).
The result needed for our purposes is the following:

\begin{lemma}\label{CL}
Assume that~$u$ is a solution of~\eqref{general problem}.
Then, for any~$j=1,\dots,n$, we have that~$u_{x_j}$
satisfies 
\begin{multline*}
\int_\C
\Big[
\langle {\mathcal{B}}(y,\nabla u)\,\nabla u_{x_j}, \nabla\psi\rangle
+g_u(y,u)\,u_{x_j}\,\psi\Big]
\\ \qquad
-\int_{\partial_L\C}
\Big[ a(y,|\nabla u|)\,(\nabla u \cdot \nabla\psi)\,\nu_j
+g(y,u)\,\psi \,\nu_j\Big]\\=
\int_{\partial_B\C} f'(u)\,u_{x_j}\,\psi
-\int_{\partial \Omega\times\{0\}} f(u)\,\psi \,\nu_j
\end{multline*}
for any $\psi \in \mathcal{A}$ 
and and such that~$\psi_{x_j}\in {\mathcal{A}}$ for any~$j=1,\dots,n$.
\end{lemma}

\begin{proof} First, we observe that, for any~$j=1,\dots,n$,
\begin{equation}\label{ihknfghvbsrdtcv98765678}
\partial_{x_j} \Big( a\big(y,|\nabla u(x,y)|\big)\Big)
= \sum_{k=1}^{n+1} a_t \big(y,|\nabla u(x,y)|\big)
\frac{u_{X_k}(x,y)\,u_{X_k x_j}(x,y)}{|\nabla u(x,y)|}.\end{equation}
Here and in the sequel, we use the notation~$X=(X_1,\dots,X_{n+1}):=(x,y)\in\R^{n+1}$,
hence~$X_k=x_k$ if~$k\in\{1,\dots,n\}$ and~$X_{n+1}=y$.

As a consequence of~\eqref{ihknfghvbsrdtcv98765678}, we have that,
for any fixed~$j=1,\dots,n$ and~$m=1,\dots,n+1$,
\begin{eqnarray*}
&& \Big( a(y,|\nabla u|)\,\nabla u_{x_j}+
\partial_{x_j} \big( a(y,|\nabla u|)\big)\,\nabla u\Big)_m\\
&=& a(y,|\nabla u|)\,u_{X_m x_j}+
\partial_{x_j} \big( a(y,|\nabla u|)\big)\,u_{X_m} 
\\ &=& \sum_{k=1}^{n+1}
a(y,|\nabla u|)\,\delta_{km}\, u_{X_k x_j}
+
\sum_{k=1}^{n+1} 
a_t (y,|\nabla u|)
\frac{u_{X_k}\,\,u_{X_m}\,u_{X_k x_j} }{|\nabla u|}
\\&=& \sum_{k=1}^{n+1} {\mathcal{B}}_{km}(y,\nabla u)\,u_{X_k x_j},
\end{eqnarray*}
where we have used~\eqref{DEF:B}.
Therefore, for any~$j=1,\dots,n$,
\begin{equation}\label{k:po}
\partial_{x_j}\big( a(y,|\nabla u|)\,\nabla u\big)  ={\mathcal{B}}(y,\nabla u)\,\nabla u_{x_j}.\end{equation}
Using \eqref{k:po}, we have the equality (in $L^1$ sense)
\begin{eqnarray*}
&& a(y,|\nabla u|)\,\nabla u\cdot \nabla\psi_{x_j} = \partial_{x_j}\Big( a(y,|\nabla u|)\,\nabla u 
\cdot \nabla\psi\Big) - \langle {\mathcal{B}}(y,\nabla u)\,\nabla u_{x_j}, \nabla\psi\rangle.
\end{eqnarray*}
We wish now to use \eqref{WEAK} with $\varphi= \psi_{x_j}$ (notice that
this is possible since we are supposing that~$\psi_{x_j} \in \mathcal{A}$). 
To this aim, we use the Divergence Theorem to obtain
\begin{equation}\label{Weak1}\begin{split} &\int_\C a(y,|\nabla u|)\,\nabla u\cdot \nabla\psi_{x_j} 
\\& \qquad= \int_{\partial_L\C}a(y,|\nabla u|)\,(\nabla u \cdot \nabla\psi)\,\nu_j
-\int_\C
\langle {\mathcal{B}}(y,\nabla u)\,\nabla u_{x_j}, \nabla\psi\rangle,
\end{split}
\end{equation}
for any~$j=1,\dots,n$. In a similar way
\begin{equation}\label{PoXXX-1}
\int_\C g(y,u)\,\psi_{x_j}= 
\int_{\partial_L \C} g(y,u)\,\psi \,\nu_j- \int_\C g_u(y,u)\,u_{x_j}\,\psi.
\end{equation}
Finally, using again the Divergence Theorem,
\begin{equation}\label{PoXXX-2}
\int_{\partial_B \C} f(u)\,\psi_{x_j} = 
\int_{\partial \Omega\times\{0\}} f(u)\,\psi \,\nu_j- \int_{\partial_B\C}
f'(u)\,u_{x_j}\,\psi.
\end{equation}
Notice that~$\nu_j$ denotes the $j$ component of $\nu=(\tilde \nu,0)$
and $\tilde \nu$ is the outer unit normal vector of $\pa \Omega$ in $\R^n$.

We can now insert \eqref{Weak1}, \eqref{PoXXX-1}
and~\eqref{PoXXX-2} into~\eqref{WEAK} and obtain the desired conclusion.
\end{proof}

As a consequence of Lemma~\ref{CL}, we can test
the linearized equation against~$\psi_j:= u_{x_j} \varphi$,
where~$\varphi$ is a ``nice'' function with bounded support in~$y$:
in this case, the particular choice of the test function
and the Neumann condition provide some additional
simplifications, as stated in the following result:

\begin{corollary}\label{HJ:cor:pho}
Assume that~$u$ is a solution of~\eqref{general problem}. Then
\begin{eqnarray}
&& \sum_{j=1}^{n}
\int_\C
\Big[
\langle {\mathcal{B}}(y,\nabla u)\,\nabla u_{x_j}, 
\nabla u_{x_j}\rangle\,\varphi+
\langle {\mathcal{B}}(y,\nabla u)\,
\nabla u_{x_j}, \nabla\varphi\rangle\,u_{x_j}\Big]  \notag\\
&&\qquad
+\int_\C g_u(y,u)\,|\nabla_x u|^2\,\varphi  \label{th corol 2.8}
\\ &&\qquad\qquad  =
\int_{\partial_L\C}
a(y,|\nabla u|)\,\big(\nabla u \cdot \partial_\nu( \nabla u)\big)\,\varphi
+ \int_{\partial_B\C} f'(u)\,|\nabla_x u|^2\varphi  \notag
\end{eqnarray}
for any $\varphi \in C^1(\overline{\C})$
with bounded support in~$y$ and such that~$\varphi_{x_j}\in {\mathcal{A}}$,
for any~$j=1,\dots,n$.\end{corollary}

\begin{proof}
We use $\psi_j:=u_{x_j} \varphi$ as test function in Lemma \ref{CL}, observing that this is possible by \eqref{hp su u-2} and the assumptions on $\varphi$, and then we sum over $j=1,\dots,n$.
First of all, for any~$x\in\partial\Omega$
and~$y>0$, if~$\psi_j:= u_{x_j} \varphi$ then on $\pa_L \C$
\begin{equation*}
\sum_{j=1}^n \psi_j \,\nu_j = \nabla_x u \cdot \nu \, \varphi=  \nabla u\cdot \nu \,\varphi=0,
\end{equation*}
where we used the fact that the last component of $\nu$ is $0$ on $\pa_L \mathcal{C}$. Since the normal along~$\partial_L\C$
coincides with the one along~$\partial\Omega$ by projection,
we deduce from~\eqref{hp su u-2} that 
$$\sum_{j=1}^n \psi_j \nu_j=0$$
also on~$\partial\Omega\times\{0\}$. These considerations imply that
\begin{equation}\label{P-eRF-1}
\begin{split}
\sum_{j=1}^n\int_{\partial_L\C} g(y,u)\,\psi_j \,\nu_j=0
\quad {\mbox{and }}\quad \sum_{j=1}^n
\int_{\partial \Omega\times\{0\}} f(u)\,\psi_j \,\nu_j=0.
\end{split}
\end{equation}
We also observe that
\begin{align*}
a(y,|\nabla u|)\, &(\nabla u \cdot \nabla\psi_j)\, \nu_j
\\
& = a(y,|\nabla u|)\,(\nabla u \cdot \nabla u_{x_j})\,\varphi\,\nu_j
+ a(y,|\nabla u|)\,(\nabla u \cdot \nabla\varphi)\, u_{x_j}\, \nu_j
\end{align*}
on~$\partial_L\C$, so that, using again the homogeneous
Neumann condition,
\begin{equation}\label{P-eRF-2}
\sum_{j=1}^n\int_{\partial_L\C}
a(y,|\nabla u|)\,(\nabla u \cdot \nabla\psi_j)\,\nu_j =
\sum_{j=1}^n\int_{\partial_L\C}
a(y,|\nabla u|)\,(\nabla u \cdot \nabla u_{x_j})\,\varphi\,\nu_j.
\end{equation}
Plugging~\eqref{P-eRF-1} and \eqref{P-eRF-2}
into the formula in Lemma~\ref{CL}, the thesis follows.
\end{proof}

A refinement of Corollary~\ref{HJ:cor:pho}, under additional integrability assumptions,
goes as follows:

\begin{corollary}\label{HJ:cor:pho-2}
Let $\Omega$ be a convex domain, and let~$u$ be a solution of~\eqref{general problem}, satisfying the integrability assumption \eqref{integrabilty assumption}. Then
\begin{eqnarray*}
&& \sum_{j=1}^{n}
\int_\C
\langle {\mathcal{B}}(y,\nabla u)\,\nabla u_{x_j},
\nabla u_{x_j}\rangle +\int_\C g_u(y,u)\,|\nabla_x u|^2\,
\\ &&\qquad\qquad  =
\int_{\partial_L\C}
a(y,|\nabla u|)\,\big(\nabla u \cdot \partial_\nu( \nabla u)\big)
+ \int_{\partial_B\C} f'(u)\,|\nabla_x u|^2.
\end{eqnarray*}
\end{corollary}

\begin{proof} We take~$\phi:=\phi_R(y)$ to be a smooth nonnegative
function such that~$\phi_R(y)=1$
if~$y\in[0,R]$, $\phi_R(y)=0$ if~$y\in[2R,+\infty)$
and~$|\phi'_R|\le 10/R$. We apply Corollary~\ref{HJ:cor:pho}
and we send~$R\to+\infty$, using assumption~\eqref{integrabilty assumption}. More precisely, by~\eqref{DEF:B} and~\eqref{ELLIPTICITY:2}, we can bound~$
|{\mathcal{B}}(y,\nabla u)|$
by~$a(y,|\nabla u|)$, up to multiplicative constants;
therefore
$$
|{\mathcal{B}}(y,\nabla u)|\, \big( |\nabla_x u|^2 + |D^2_x u|^2 +|\nabla_x u_y|^2\big)
\in L^1(\C)$$
thanks to~\eqref{integrabilty assumption}, and this allows
us to pass to the limit as~$R\to+\infty$ in the first term on the left hand side in \eqref{th corol 2.8}. As far as the second term, we use the fact that $g_u(y,u) |\nabla_xu|^2 \in L^1(\C)$ and argue in a similar way. Finally, for the first term on the right hand side of \eqref{th corol 2.8} we apply the Monotone Convergence Theorem, observing that, thanks to the convexity of $\Omega$, Lemma \ref{abbiamo} implies that
\[
a(y,|\nabla u|) \left(\nabla u \cdot \pa_\nu(\nabla u) \right) \le 0 \quad \text{on $\pa_L \C$}. \qedhere
\]
\end{proof}

\section{A Poincar\'e-type geometric inequality: proof
of Theorem~\ref{TH:POI}}\label{SEC:POINC-2}

In this section we prove the 
geometric inequality of Poincar\'e type stated in Theorem~\ref{TH:POI}.

\begin{proof}[Completion of the proof of Theorem~\ref{TH:POI}]
We use Corollary~\ref{HJ:cor:pho} with~$\varphi:=\psi^2$
and $\psi \in C^1(\overline{\C})$ with bounded support in $y$
and such that~$\psi_{x_j}\in{\mathcal{A}}$ for any~$j=1,\dots,n$.
In this way, we have that
\begin{equation}\label{result from equation}
\begin{split}
& \sum_{j=1}^{n}
\int_\C
\Big[
\langle {\mathcal{B}}(y,\nabla u)\,\nabla u_{x_j},
\nabla u_{x_j}\rangle\,\psi^2+2
\langle {\mathcal{B}}(y,\nabla u)\,
\nabla u_{x_j}, \nabla\psi\rangle\,u_{x_j}\,\psi\Big]\\
&\qquad
+\int_\C g_u(y,u)\,|\nabla_x u|^2\,\psi^2
\\ &\qquad\qquad  =
\int_{\partial_L\C}
a(y,|\nabla u|)\,\big(\nabla u \cdot \partial_\nu( \nabla u)\big)\,\psi^2
+ \int_{\partial_B\C} f'(u)\,|\nabla_x u|^2\psi^2.
\end{split}
\end{equation}
Now we use the fact that $u$ is stable, and we choose~$|\nabla_x u| \psi$
as 
test function in the definition of stability~\eqref{stability}
(we observe that this choice is admissible, thanks to~\eqref{hp su u-2}):
in this way, we conclude that
\begin{equation}\label{jh:0kl}
\begin{split}
& \int_{\C} \langle \mathcal{B}(y,\nabla u)
\nabla |\nabla_x u|,\, \nabla |\nabla_x u|\rangle \,\psi^2
+\int_{\C}
\langle \mathcal{B}(y,\nabla u)
\nabla \psi,\,\nabla \psi\rangle\,|\nabla_x u|^2 \\&\qquad+
\int_{\C}
2\langle \mathcal{B}(y,\nabla u)
\nabla |\nabla_x u|,\,\nabla \psi\rangle\,|\nabla_x u|\,\psi
\\&\qquad+ \int_{\C}g_u(y,u)\,|\nabla_x u|^2\,\psi^2 -
\int_{\pa_B \C} f'(u) \,|\nabla_x u|^2\,\psi^2\ge0.
\end{split}\end{equation}
It is convenient to observe that
$$ |\nabla_x u|\,\nabla|\nabla_x u|
=\frac12 \nabla |\nabla_x u|^2 
=\frac12 \sum_{j=1}^n \nabla u_{x_j}^2
=\sum_{j=1}^n u_{x_j} \nabla u_{x_j},$$
and so we can rewrite~\eqref{jh:0kl} as
\begin{equation*}
\begin{split}
& \int_{\C} \langle \mathcal{B}(y,\nabla u)
\nabla |\nabla_x u|,\, \nabla |\nabla_x u|\rangle \,\psi^2
+\int_{\C}
\langle \mathcal{B}(y,\nabla u)
\nabla \psi,\,\nabla \psi\rangle\,|\nabla_x u|^2 \\&\qquad+
\sum_{j=1}^n\int_{\C} 2
\langle \mathcal{B}(y,\nabla u)
\nabla u_{x_j},\,\nabla \psi\rangle\,u_{x_j}\,\psi
\\&\qquad+ \int_{\C}g_u(y,u)\,|\nabla_x u|^2\,\psi^2 -
\int_{\pa_B \C} f'(u) \,|\nabla_x u|^2\,\psi^2\ge0.
\end{split}\end{equation*}
This expression and \eqref{result from equation}
have three terms in common, which can be simplified
appropriately,
thus establishing~\eqref{POI:FO}.
\end{proof}

\section{Classification in convex domains I:
proof of Theorem~\ref{thm: symmetry convex 1}}\label{KJ:8a9}

We claim that
\begin{equation}\label{ihkAFYBJOIHKGBJ}
\begin{split}&
{\mbox{if~$\nabla_xu(x_o,y_o)\ne0$, then
each connected component of the
level set}} \\&{\mbox{of the function~$\Omega\ni x\mapsto u(x,y_o)$
must be}} \\&{\mbox{an $(n-2)$-dimensional hyperplane intersected~$\Omega$
}}\end{split}
\end{equation}
and that
\begin{equation}\label{PO:PO2:VX56}
\nabla u \cdot \partial_\nu( \nabla u)=0
{\mbox{ on }}\partial_L\C.
\end{equation}
Let~$R>10$ (to be taken arbitrarily large in the sequel).
We consider a smooth function~$\tau_R:[0,+\infty)\to [0,1]$ supported in~$[\sqrt{R},\,R]$, 
such that~$\tau_R=1$ in~$[\sqrt{R}+1,\,R-1]$
and~$|\nabla\tau_R|\le10$. For any~$y\ge0$, we define
$$ \psi_R(y):= \int_y^R \frac{\tau_R(\zeta)}{\zeta}\,d\zeta.$$
Since~$\tau_R$ vanishes in~$[0,\sqrt{R}]$ we have that~$\psi_R$
is smooth, continuous in~$[0,+\infty)$ and
\begin{equation}\label{HHJ}
\begin{split}
\psi_R(y) &= \int_{\sqrt{R}}^R \frac{\tau_R(\zeta)}{\zeta}\,d\zeta
\ge \int_{\sqrt{R}+1}^{R-1} \frac{\tau_R(\zeta)}{\zeta}\,d\zeta
\\&=
\int_{\sqrt{R}+1}^{R-1} \frac{d\zeta}{\zeta}
=\log \frac{R-1}{\sqrt{R}+1}\ge 
\log\frac{\sqrt{R}}{2},
\end{split}\end{equation}
for any~$y\in[0,\sqrt{R}]$, as long as~$R$ is sufficiently large. In addition, since~$\tau_R$ 
vanishes also in~$[R,+\infty)$, we have that~$\psi_R(y)=0$
for any~$y\ge R$, hence~$\psi_R$ is compactly supported in~$[0,+\infty)$. Finally, $\psi_R \in C^2(\overline{\C})$ by the choice of $\tau_R$. As a consequence, we can use~$\psi_R$ as a test function in~\eqref{POI:FO}: this yields
\begin{equation}\label{POI:FO:BIS}
\begin{split}
& \int_\C\left[ \sum_{j=1}^{n}
\langle {\mathcal{B}}(y,\nabla u)\,\nabla u_{x_j},
\nabla u_{x_j}\rangle
-\langle \mathcal{B}(y,\nabla u)
\nabla |\nabla_x u|,\, \nabla |\nabla_x u|\rangle\right] \,\psi_R^2
\\ &\qquad-\int_{\partial_L\C}
a(y,|\nabla u|)\,\big(\nabla u \cdot \partial_\nu( \nabla u)\big)\,\psi_R^2
\le\int_{\C}\langle \mathcal{B}(y,\nabla u)
\nabla \psi_R,\,\nabla \psi_R\rangle\,|\nabla_x u|^2 .
\end{split}\end{equation}
Now we use~\eqref{key-eq-EUCLIDEAN-2}, \eqref{OIhh}, \eqref{HHJ}, and the fact that $\nabla_x u$ is constant almost everywhere on $\{\nabla_x u=0\}$, 
by Stampacchia's Theorem,
to see that
\begin{equation}\label{OhgFF67:001}
\begin{split}
&\int_\C\left[ \sum_{j=1}^{n}
\langle {\mathcal{B}}(y,\nabla u)\,\nabla u_{x_j},
\nabla u_{x_j}\rangle
-\langle \mathcal{B}(y,\nabla u)
\nabla |\nabla_x u|,\, \nabla |\nabla_x u|\rangle\right] \,\psi_R^2 \ge\\
& \left(\log\frac{\sqrt{R}}{2}\right)^2
\int_{\Omega\times(0,\sqrt{R})}\left[ \sum_{j=1}^{n}
\langle {\mathcal{B}}(y,\nabla u)\,\nabla u_{x_j},
\nabla u_{x_j}\rangle
-\langle \mathcal{B}(y,\nabla u)
\nabla |\nabla_x u|,\, \nabla |\nabla_x u|\rangle\right] 
\end{split}
\end{equation}
and
\begin{equation}\label{OhgFF67:002}
\begin{split}
&-\int_{\partial_L\C}
a(y,|\nabla u|)\,\big(\nabla u \cdot \partial_\nu( \nabla u)\big)\,\psi_R^2\\
\ge\,& -\left(\log\frac{\sqrt{R}}{2}\right)^2\,
\int_{\partial_L\C\cap\{ y\in(0,\sqrt{R})\}}
a(y,|\nabla u|)\,\big(\nabla u \cdot \partial_\nu( \nabla u)\big).\end{split}
\end{equation}
In this way we can estimate the left hand side in \eqref{POI:FO:BIS}. As far as the right hand side is concerned, by~\eqref{ELLIPTICITY:2} 
we obtain
\begin{multline*}
\int_{\C}\langle \mathcal{B}(y,\nabla u)
\nabla \psi_R,\,\nabla \psi_R\rangle\,|\nabla_x u|^2
\le C_1
\int_{\C}a(y,|\nabla u|)
\,|\nabla \psi_R|^2\,|\nabla_x u|^2 \\
\le C_2\,\int_{\Omega\times(\sqrt{R},R)} a(y,|\nabla u|)
\,\frac{|\tau_R(y)|^2\,|\nabla_x u|^2}{y^2} \le C_2\,\int_{\Omega\times(\sqrt{R},R)} a(y,|\nabla u|)
\,\frac{|\nabla_x u|^2}{y^2} ,
\end{multline*}
for some~$C_1$, $C_2>0$. So, by using Lemma~\ref{herein}
with~$h:=a(y,|\nabla u|)\,|\nabla_x u|^2$
and recalling~\eqref{energy bound}, we conclude that
\begin{equation}\label{OhgFF67:003}
\begin{split}
&\int_{\C}\langle \mathcal{B}(y,\nabla u)
\nabla \psi_R,\,\nabla \psi_R\rangle\,|\nabla_x u|^2\\
\le\,& C_3 \int_{\sqrt{R}}^R
\left[\int_{\Omega\times(0,t)}
a(y,|\nabla u|)\,|\nabla_x u|^2
\right]\,\frac{dt}{t^3}+
\frac{C_3}{R^2}
\int_{\Omega\times(0,R)}
a(y,|\nabla u|)\,|\nabla_x u|^2
\\ \le\,& C_4
\int_{\sqrt{R}}^R \frac{dt}{t}+
C_4 \le\, C_5 \log\sqrt R,
\end{split}\end{equation}
provided that $R$ is large
enough, for some constants~$C_3$, $C_4$, $C_5>0$.
Thus we insert~\eqref{OhgFF67:001}, \eqref{OhgFF67:002}
and~\eqref{OhgFF67:003} into~\eqref{POI:FO:BIS},
we divide by~$\left(\log\frac{\sqrt{R}}{2}\right)^2$ and we deduce
\begin{eqnarray*}
&& \int_{\Omega\times(0,\sqrt{R})}\left[ \sum_{j=1}^{n}
\langle {\mathcal{B}}(y,\nabla u)\,\nabla u_{x_j},
\nabla u_{x_j}\rangle
-\langle \mathcal{B}(y,\nabla u)
\nabla |\nabla_x u|,\, \nabla |\nabla_x u|\rangle\right]\\
&&\qquad-
\int_{\partial_L\C\cap\{ y\in(0,\sqrt{R})\}}
a(y,|\nabla u|)\,\big(\nabla u \cdot \partial_\nu( \nabla u)\big)
\,\le\, \frac{C_6 \log R}{\left(\log\frac{\sqrt{R}}{2}\right)^2}.
\end{eqnarray*}
Since the latter term is infinitesimal as~$R\to+\infty$, the previous estimate implies
\begin{eqnarray*}
&& \int_{\C}\left[ \sum_{j=1}^{n}
\langle {\mathcal{B}}(y,\nabla u)\,\nabla u_{x_j},
\nabla u_{x_j}\rangle
-\langle \mathcal{B}(y,\nabla u)
\nabla |\nabla_x u|,\, \nabla |\nabla_x u|\rangle\right]\\
&&\qquad-
\int_{\partial_L\C}
a(y,|\nabla u|)\,\big(\nabla u \cdot \partial_\nu( \nabla u)\big)
\,\le\, 0,
\end{eqnarray*}
and as a consequence
\begin{equation}
\label{PO:PO1} \left[ \sum_{j=1}^{n}
\langle {\mathcal{B}}(y,\nabla u)\,\nabla u_{x_j},
\nabla u_{x_j}\rangle
-\langle \mathcal{B}(y,\nabla u)
\nabla |\nabla_x u|,\, \nabla |\nabla_x u|\rangle\right]=0
{\mbox{ in }}\C
\end{equation}
and
\begin{equation*}
a(y,|\nabla u|)\,\big(\nabla u \cdot \partial_\nu( \nabla u)\big)=0
{\mbox{ on }}\partial_L\C,\end{equation*}
thanks to~\eqref{key-eq-EUCLIDEAN-2} and~\eqref{OIhh}.
This establishes~\eqref{PO:PO2:VX56} (recall also~\eqref{ELLIPTICITY}).

Also, by \eqref{PO:PO1}
and Corollary~\ref{Ut:OI},
we obtain that~\eqref{ihkAFYBJOIHKGBJ} holds.
\medskip

Now we use that~$\partial\Omega$ 
has positive principal curvatures.
For this, we claim that
\begin{equation}\label{JU:tp}
{\mbox{$u$ is constant along~$\partial\Omega\times\{\bar y\}$,}}\end{equation}
for any fixed~$\bar y>0$.
To prove it assume by contradiction
that~$u(p,\bar y)\ne u(q,\bar y)$ for some~$p$, $q\in \partial\Omega$.
By Lemma~\ref{con:PROI}, we know that we can connect $p$ to~$q$
with a continuous path~$\sigma:[0,1]\to\partial\Omega$.
Let~$\zeta(t):= u(\sigma(t),\,\bar y)$.
Then $\zeta(0) \neq \zeta(1)$, and
therefore there exists~$\bar t\in (0,1)$ such that~$\dot\zeta(\bar t)\ne0$.
That is
\begin{equation} \label{OUGFh} 0\ne
\dot\zeta(\bar t) = \nabla_x u(\sigma(\bar t)
,\,\bar y) \cdot \dot\sigma(\bar t).\end{equation}
We let~$\bar x:=\sigma(\bar t)$.
Up to a change of coordinates, we may suppose that the exterior
normal of~$\partial\Omega$ at~$\bar x$
coincides with~$-e_n$,
hence, near~$\bar x$ the domain~$\Omega$ can
be written in normal coordinates as the epigraph of a function~$\gamma\in
C^2(\R^{n-1})$. The fact that the principal curvatures of~$\partial\Omega$
are positive implies that 
\begin{equation}\label{PO:k67YY}
{\mbox{the Hessian of~$\gamma$
is positive definite}}.\end{equation}
On the other hand, by~\eqref{PO:PO2:VX56} and~\eqref{key-eq-EUCLIDEAN},
we have that
$$ 0=-\nabla u(\bar x,\bar y) \cdot \partial_\nu
\left(\nabla u(\bar x,\bar y)\right) =  \sum_{i,j=1}^{n-1}\gamma_{x_i x_j}
(\bar x')\,u_{x_i}(\bar x,\bar y)\,u_{x_j}(\bar x,\bar y).$$
This and~\eqref{PO:k67YY} give that~$u_{x_i}(\bar x,\bar y)=0$
for any~$i=1,\dots,n-1$. 
By the Neumann condition and the choice of the coordinate system,
we also know that~$u_{x_n}=-\partial_\nu
u=0$ in $(\bar x,\bar y)$. 
Hence~$\nabla_x u(\bar x,\bar y)=0$, in contradiction with~\eqref{OUGFh}; this proves~\eqref{JU:tp}.

Now we show that
\begin{equation}\label{JU:tp:2}
{\mbox{$u$ is constant in $\Omega\times\{\bar y\}$,}}\end{equation}
for any fixed~$\bar y>0$. For this, 
we argue by contradiction and we assume that this is not true:
as a consequence, we have that
\begin{equation}\label{JU:tp:2:CONR}
\{ x\in\Omega {\mbox{ s.t. }}\nabla_x u(x,\bar y)\ne0\}\ne\varnothing.
\end{equation}
We let~$c(\bar y)$ be the value
attained by~$u$ along~$\partial\Omega\times\{\bar y\}$,
as given by~\eqref{JU:tp}.

We also take an arbitrary point~$x_0\in\Omega$
for which~$\nabla_x u(x_0,\bar y)\ne0$ (here, we are using~\eqref{JU:tp:2:CONR}
to say that such point exists). We let~$L(x_0)$
be the 
connected component of the
level set of~$u(\cdot,\bar y)$ in~$\Omega$
which passes through~$x_0$.
In view of~\eqref{ihkAFYBJOIHKGBJ}, we know that
\begin{equation}\label{9ikHAKAKKAK67890987654}
L(x_0)=\{ \omega\cdot(x-x_0)=0\}\cap\Omega,\end{equation}
for a suitable~$\omega\in S^{n-1}$, possibly depending
on~$x_0$ and~$\bar y$.
We also consider a vector~$\varpi$ orthogonal to~$\omega$
(of course, $\varpi$ may also depend
on~$x_0$ and~$\bar y$).

Then,
we consider the straight line
$$ \{ x_0+ \varpi t, \; t\in\R\}.$$
Since the domain~$\Omega$ is bounded, such a line must
intersect somewhere the boundary of~$\Omega$, i.e.
there exists~$t_0$
such that~$x_0+\varpi t_0\in \partial\Omega$.
Therefore, by~\eqref{JU:tp},
\begin{equation}\label{Kj7855:A}
u\big(x_0+\varpi t_0,\,\bar y\big)=c(\bar y).\end{equation}
On the other hand, by~\eqref{9ikHAKAKKAK67890987654},
\begin{equation*}
u\big(x_0+\varpi t_0,\,\bar y\big)=u\big(x_0,\,\bar y\big).\end{equation*}
This and~\eqref{Kj7855:A} give that
$$ u(x_0,\bar y)=c(\bar y).$$
Since this holds for any point~$x_0\in\Omega$
for which~$\nabla_x u(x_0,\bar y)\ne0$,
we have established that
$$ {\mbox{$u(x,\bar y)=c(\bar y)$ for any $x\in\Omega\cap
\{ \nabla_x u(\cdot,\bar y)\ne0\}$.}}$$
Since the above identity also holds on~$\partial\Omega$
and since~$u$ is constant in each component of~$
\Omega\cap
\{ \nabla_x u(\cdot,\bar y)=0\}$, we obtain that~$u(x,\bar y)=c(\bar y)$ for any $x\in\Omega$,
hence~$\nabla_xu$ vanishes identically in~$\Omega$.

This is in contradiction with~\eqref{JU:tp:2:CONR}; hence,
we have established~\eqref{JU:tp:2},
and thus completed the proof of
Theorem~\ref{thm: symmetry convex 1}.

\section{Classification in convex domains II:
proof of Theorem~\ref{P-CA}}\label{HG:SEC}

Now we address the proof of Theorem~\ref{P-CA}. For this goal,
we need a detailed study of functions which attain
the minimum of the stability functional~$I$ 
introduced in~\eqref{stability}. After this,
we will use the geometric observations
exposed in Section~\ref{sec: preliminaries}
and suitable test functions to complete the proof
of Theorem~\ref{P-CA}.

\subsection{Rigidity of minimal solutions}

In this part, we study the rigidity properties
of the minimizers of the stability functional~$I$,
as defined in~\eqref{stability}. To this aim, we introduce 
\begin{equation}\label{JH:AHJK}
{\mathcal{A}}^*:= \left\{ \varphi\in W^{1,1}_{\rm loc}(\C) \left| \begin{array}{l}
a(y,|\nabla u|)\,\Big(\varphi^2+|\nabla\varphi|^2\Big)+ |g_u(y,u)| \varphi^2 \in L^1(\C) 
\\ 
 {\mbox{and }} \varphi|_{\Omega \times \{0\}} \in L^2(\Omega)
\end{array}\right.\right\}.
\end{equation}
With respect to the space~$\mathcal{A}$ defined in \eqref{DEF:A}, we replace the requirement that $\varphi$ has bounded support in $y$ with an integrability condition. {F}rom now on,
we assume that~$u$ is a stable solution of~\eqref{general problem},
according to~\eqref{stability}.

\begin{lemma}\label{lem: density stability}
Let $I$ be defined by \eqref{stability}. Then $I(\varphi) \ge 0$ for every $\varphi \in \mathcal{A}^*$.
\end{lemma}
\begin{proof}
We proceed by approximation in the following way: let $\tau_R$ be a smooth function such that
\begin{eqnarray*}
\tau_R(t)=
\begin{cases}
1, {\mbox{ if }} 0<t\le R, \\
0,  {\mbox{ if }} t\ge 2R,
\end{cases}
\end{eqnarray*}
and $|\tau'_R|\le C/R$ for some $C>0$. 
Given $\varphi \in \mathcal{A}^*$, the function $\varphi_R(x,y):= 
\varphi(x,y)\,\tau_R(y)$ belongs to $\mathcal{A}$, and hence can be used as a
test function in the stability assumption \eqref{stability}. 
Therefore
\begin{equation}\label{PANaHBKKH}
I(\varphi_R)\ge0.
\end{equation}
In order to obtain the desired result, we aim at passing to the limit as $R \to +\infty$. The details go as follows.
First of all, we recall~\eqref{DEF:B} and \eqref{ELLIPTICITY:2}, to point out that
$$ \big|\mathcal{B}(y,\eta)\big| \le
C_1 \Big( a(y,|\eta|) + a_t(y,|\eta|)\,|\eta|
\Big) \le C_2 \,a(y,|\eta|),$$
for some~$C_1$, $C_2>0$.
As a consequence
\begin{equation}\label{ipotesi strana-2}
\lim_{R \to +\infty} \frac{1}{R^2}
\int_{\Omega \times (R,2R)} \big|{\mathcal{B}}(y,\nabla u)\big|
\,\varphi^2 \le 
\lim_{R \to +\infty} \frac{C_2}{R^2} \int_{\C} a(y,|\nabla u|)\,
\,\varphi^2 =0,
\end{equation}
where we used the integrability conditions in \eqref{JH:AHJK}.

Now, from \eqref{PANaHBKKH}
we infer that
\begin{equation}\label{P0hj}
\begin{split}
0 
&\le\int_{\C} \langle \mathcal{B}(y,\nabla u)  
\nabla \varphi,\nabla \varphi\rangle \,\tau_R^2
+ \int_{\C} \langle \mathcal{B}(y,\nabla u)  
\nabla \tau_R,\nabla \tau_R\rangle \,\varphi^2\\
&\quad + 2\int_{\C} \langle \mathcal{B}(y,\nabla u)  \nabla \varphi,
\nabla \tau_R\rangle \,\varphi\,\tau_R
+ g_u(y,u)\varphi^2\tau_R^2 -\int_{\pa_B \C} f'(u) \varphi^2.\end{split}\end{equation}
Let~$\epsilon>0$
(to be taken arbitrarily small in the sequel); we use a weighted H\"older inequality to observe that
\begin{multline*}
\int_{\C} \langle \mathcal{B}(y,\nabla u)  \nabla \varphi,
\nabla \tau_R\rangle \,\varphi\,\tau_R \\
\le \epsilon \int_{\C} \langle \mathcal{B}(y,\nabla u)  \nabla \varphi,
\nabla \varphi\rangle \,\tau_R^2
+C_\epsilon
\int_{\C} \langle \mathcal{B}(y,\nabla u)  \nabla \tau_R,
\nabla \tau_R\rangle \,\varphi^2
\end{multline*}
for some~$C_\epsilon>0$. Plugging this into~\eqref{P0hj},
we obtain 
\begin{eqnarray*}
0 &\le& (1+\epsilon)
\int_{\C} \langle \mathcal{B}(y,\nabla u)  
\nabla \varphi,\nabla \varphi\rangle \tau_R^2
+\frac{1+C_\epsilon}{R^2} \int_{\Omega \times (R,2R)} \big| \mathcal{B}(y,\nabla u)\big|\,\varphi^2\\
&&\quad + 
\int_{\C} g_u(y,u) \varphi^2\tau_R^2 -\int_{\pa_B \C} f'(u) \varphi^2
.\end{eqnarray*}
Recalling~\eqref{ipotesi strana-2} and the definition of $\mathcal{A}^*$, we can pass to the limit as $R \to +\infty$, concluding that
$$ 0 \le (1+\epsilon)
\int_{\C} \langle \mathcal{B}(y,\nabla u)
\nabla \varphi,\nabla \varphi\rangle +
\int_{\C} g_u(y,u) \varphi^2 -\int_{\pa_B \C} f'(u) \varphi^2
.$$
Since $\epsilon>0$ has been arbitrarily chosen, the thesis follows.
\end{proof}

In light of the previous result, we write that $\varphi \in \mathcal{A}^*$ is a \emph{minimizer} for $I$ if $I(\varphi) \le 0$
(equivalently, by Lemma~\ref{lem: density stability},
if $I(\varphi) =0$). First, we show that
minimizers satisfy a suitable reaction-diffusion equation,
both in the weak and in the strong sense:

\begin{lemma}\label{L1}
Assume
that~$I(\varphi)\le 0$, for some~$\varphi\in {\mathcal{A}^*}$. Then
$$ \int_{\C} \langle {\mathcal{B}}(y,\nabla u) \nabla\varphi,
\nabla\zeta\rangle
+\int_{\C} g_u(y,u)\varphi\zeta - \int_{\partial_B\C} f'(u)\varphi\zeta=0,$$
for any $\zeta
\in {\mathcal{A}^*}$.\end{lemma}

\begin{proof}
For any~$\epsilon\in\R$
and any test function~$\zeta$, we have that~$I(\varphi+\epsilon\zeta)
\ge0\ge I(\varphi)$, and therefore, dividing by $\epsilon$
and sending $\epsilon\to0$, we obtain the desired result.
\end{proof}

\begin{lemma}
Assume also
that~$I(\varphi)\le 0$, for some~$\varphi\in \mathcal{A}^* \cap C^2(\C)\cap C^1(
\overline\Omega\times[\alpha,\beta])$, for any~$\beta>\alpha>0$
Then, $\varphi$ is a solution of
\begin{eqnarray}
&&\label{bou-EQ}
\sum_{i,j=1}^{n+1} 
{\mathcal{B}}_{ij}(y,\nabla u) 
\partial^2_{X_i X_j}\varphi+
\partial_{X_i}
{\mathcal{B}}_{ij}(y,\nabla u)
\partial_{X_j}\varphi - g_u(y,u)\varphi=0
{\mbox{ in }}\C,
\\
&&\label{bou-00}  {\mbox{with
$\partial_\nu \varphi=0$
on $\partial_L\C$,}}
\end{eqnarray}
where we recall that $X=(x,y) \in \R^{n+1}$.
\end{lemma}

\begin{proof}
We use Lemma \ref{L1}. Indeed, by taking~$\zeta$ supported 
inside~$\C$, we obtain \eqref{bou-EQ}.
By taking~$\zeta$ supported
near any given point of~$\partial_L\C$, we conclude that also \eqref{bou-00} holds.
\end{proof}

With this, we are in the position of obtaining
a strict sign for nonnegative minimizers, up to the boundary,
in the spirit of a strict comparison principle,
according to
the following result:

\begin{corollary}\label{LK:87hhj}
Assume that $I(\varphi)\le 0$ for $\varphi\in 
{\mathcal{A}}^* \cap C^2(\C)\cap C^1(
\overline\Omega\times[\alpha,\beta])$, for any~$\beta>\alpha>0$. Assume in addition that~$\varphi(x,y)\ge0$ for any~$(x,y)\in\C$.

Then either~$\varphi(x,y)>0$ for any~$x\in\overline\Omega$
and~$y>0$, or~$\varphi(x,y)=0$ for any~$x\in\Omega$ and any~$y>0$.
\end{corollary}

\begin{proof}
Suppose that~$\varphi(x_o,y_o)=0$ for some~$x_o\in\partial\Omega$
and~$y_o>0$. Then, we look at the equation
satisfied by~$\varphi$ in~$\Omega_{\alpha,\beta}$,
where~$\Omega_{\alpha,\beta}$ is a smooth domain that contains~$\Omega
\times (\alpha,\beta)$ and is contained
in~$\Omega
\times \left(\alpha/2,\beta\right)$
with~$0<\alpha<y_o<\beta$.
Indeed, we define
\begin{eqnarray*}
&& M:=\sup_{(x,y) \in \Omega_{\alpha,\beta}} \big| g_u(y,u(x,y))\big|<+\infty, \\
&& a_{ij}(x,y):= {\mathcal{B}}_{ij}(y,\nabla u(x,y)\big),\\
&& b_j(x,y):= \sum_{i=1}^{n+1}
\partial_{X_i} {\mathcal{B}}_{ij}\big(y,\nabla u(x,y)\big) \\
&& c(x,y):= -M- g_u(y,u(x,y)).
\end{eqnarray*}
Notice that~$c\le0$ and~$a_{ij}$ defines an elliptic matrix
on~$\Omega\times (\alpha,\beta)$, for fixed~$\alpha$ and~$\beta$,
thanks to Lemma~\ref{B-POS}. Moreover,
by~\eqref{bou-EQ}
$$ \sum_{i,j=1}^{n+1}
a_{ij} \partial^2_{X_i X_j}\varphi+\sum_{j=1}^{n+1} b_j
\partial_{X_j}\varphi +c \varphi = -M\varphi\le0.$$
Notice that~$a_{ij}$, $b_j$, $c\in C(\overline\Omega_{\alpha,\beta})$,
thanks to \eqref{hp su u-2}.
Also, we have that~$\varphi$ attains its minimum
in~$\overline\Omega_{\alpha,\beta}$
at~$(x_o,y_o)$. As a consequence, by the Hopf Lemma (see e.g. Corollary 1.6
in Chapter~2 of~\cite{HAN-LIN}), either~$\varphi$ vanishes identically in~$\Omega_{\alpha,\beta}$
or~$0\ne\partial_\nu\varphi(x_o,y_o)$. The latter possibility cannot hold, in light of~\eqref{bou-00},
and therefore~$\varphi$ must,
in this case,
vanish identically in the domain~$\Omega_{\alpha,\beta}$.
Since~$\alpha$ can be taken as close to~$0$ as we wish
and~$\beta$ can be taken arbitrarily large,
this implies that~$\varphi$ must
vanish everywhere in~$\Omega\times(0,+\infty)$.
\end{proof}

As a matter of fact, we can strengthen Corollary~\ref{LK:87hhj}
by removing the sign assumption on~$\varphi$.
Namely, we have that:

\begin{proposition}\label{LK:87hhj:C}
Assume
that~$I(\varphi)\le 0$ for~$\varphi\in {\mathcal{A}^*}\cap
C^2(\C)\cap C^1(
\overline\Omega\times[\alpha,\beta])$, for any~$\beta>\alpha>0$. Then, one and only one of these three possibilities holds true:
\begin{itemize}
\item $\varphi(x,y)>0$ for any~$x\in\overline\Omega$ and~$y>0$, 
\item $\varphi(x,y)<0$ for any~$x\in\overline\Omega$ and~$y>0$, 
\item $\varphi(x,y)=0$ for any~$x\in\Omega$ and any~$y>0$.
\end{itemize}\end{proposition}

\begin{proof} We claim that 
\begin{equation} \label{UI:per}
{\mbox{either $\varphi\ge0$ 
or $\varphi\le0$ in $\C$.}} \end{equation} 
To prove this, we 
consider~$\varphi^+:=\max\{ \varphi,0\}$ and~$\varphi^-:=\max\{ 
-\varphi,0\}$. We observe that~$\varphi^\pm \in 
{\mathcal{A}^*}$.
Thus,
by Lemma~\ref{L1}, 
\begin{align*}
I(\varphi^{\pm}) & = \int_\C \langle 
{\mathcal{B}}(y,\nabla u\big)\,\nabla\varphi^\pm, 
\,\nabla\varphi^\pm \rangle +\int_\C g_u(y,u)\,(\varphi^\pm)^2 
-\int_{\partial_B\C} f'(u)\,(\varphi^\pm)^2 \\
& = \int_\C \langle {\mathcal{B}}(y,\nabla u\big)\,\nabla\varphi, 
\,\nabla\varphi^\pm \rangle +\int_\C g_u(y,u)\,\varphi\,\varphi^\pm 
-\int_{\partial_B\C} f'(u)\,\varphi\,\varphi^\pm = 0.
\end{align*}
Hence, using again Lemma~\ref{L1},
we have that, for any~$\zeta\in C^\infty_0 (\C)$,
$$ \int_\C 
\langle {\mathcal{B}}(y,\nabla u\big)\,\nabla\varphi^\pm, 
\,\nabla\zeta \rangle +\int_\C g_u(y,u)\,\varphi^\pm\,\zeta =0.$$ 
Therefore, the function~$\varphi^\pm$ is a weak solution to 
\[
\div\big(\mathcal{B}(y,\nabla u)\nabla \varphi^{\pm}\big) = g_u(y,u) \varphi^{\pm} \qquad \text{in $\C$},
\]
according to 
the notation of Chapter~8 in~\cite{GT-rep-1998} (in particular, 
formula~(8.5) there is a consequence of Lemma~\ref{B-POS} here and 
formula~(8.6) applies here with~$b^i:=c^i:=0$ and~$d:=g_u(\cdot,u)$). 
Consequently, by Theorem~8.20 in~\cite{GT-rep-1998}, for any point~$X_o
\in\C$ and any~$R>0$ such that~$B_{4R}(X_o)\subset\C$,
$$ \sup_{B_R(X_o)} \varphi^\pm \le C_R\,\inf_{B_R(X_o)} \varphi^\pm,$$
for some~$C_R>0$.
This implies that if~$\varphi^\pm$ vanishes somewhere in~$\C$,
then it must vanish identically in~$\C$,
thus completing the proof of~\eqref{UI:per}.

Thanks to~\eqref{UI:per} we can now exploit
Corollary~\ref{LK:87hhj} (applying this to~$\varphi$
if~$\varphi\ge0$ or to~$-\varphi$ if~$\varphi\le0$).
{F}rom this, we obtain the desired result.
\end{proof}

With this preparatory work, we are now in the position
of finishing the proof of Theorem~\ref{P-CA}:

\begin{proof}[Completion of the proof of Theorem~\ref{P-CA}]
By \eqref{hp su u-2} and the integrability assumption \eqref{integrabilty assumption}, we have that~$u_{x_j} \in \mathcal{A}^*$ for every $j=1,\dots,n$. Thus, by Lemma \ref{lem: density stability}, we have that~$I(u_{x_j}) \ge 0$.
On the other hand,
using Corollary~\ref{HJ:cor:pho-2}
and Lemma~\ref{abbiamo}, 
\[
\sum_{j=1}^n I (u_{x_j}) =
\int_{\partial_L\C}
a(y,|\nabla u|)\,\big(\nabla u \cdot \partial_\nu( \nabla u)\big)
\le0,
\]
so that necessarily~$I(u_{x_j})=0$
for any~$j=1,\dots,n$. Therefore, by Proposition~\ref{LK:87hhj:C}, we deduce that
either~$u_{x_j}$ never vanishes in~$\overline\Omega\times(0,+\infty)$,
or~$u_{x_j}$ vanishes identically in~$\C$. But the first possibility cannot occur: 
to see this, let us slide a hyperplane normal to~$e_j$
till it touches~$\partial\Omega$ at some point~$x^\star_j$.

By construction, the normal of~$\partial\Omega$ at~$x^\star_j$
is~$e_j$, hence the homogeneous Neumann condition $\pa_\nu u=0$ on $\pa_L \mathcal{C}$
implies that~$
0=\partial_\nu u(x^\star_j, 1)=u_{x_j}(x^\star_j, 1)$. This shows that the first above-mentioned
possibility cannot occur, and as a consequence~$u_{x_j}$ vanishes
identically in~$\C$. Since this is valid for any~$j=1,\dots,n$,
this implies that~$u$ does not depend on~$x$.
\end{proof}

\section{Classification of stable solutions for convex/concave
nonlinearities: proof of Theorems~\ref{TH:CC:1}}\label{LK:COCO-0}

We now address the case in which~$f$ satisfies suitable
convexity or concavity assumption. In this setting, we need some
preliminary work in order to detect the sign of the nonlinearity
at the maximum or at the minimum. We stress that, in this section, we always suppose that $u$ is a bounded stable solution to \eqref{general problem}, and that assumptions \eqref{LO} and \eqref{NUOVE} are in force.

\subsection{Detecting the sign of the nonlinearity}

A classical tool in partial differential equations
is the use of various forms of maximum and comparison principles
in order to check the sign of the nonlinearities
at the points in which solutions of elliptic equations attain
their extremal values. In our setting, we adapt these
type of strategies, with the aim of detecting the sign
of~$f$ at the extremal values
for solutions of the reaction-diffusion equation \eqref{general problem}.
This goal will be accomplished in Corollary~\ref{C:PT}.
For this, we need an auxiliary result, which locates
the extremal values at the bottom boundary of the domain.

\begin{lemma}\label{0oPPy}
Let $v(x):=u(x,0)$.
Then
$$ \min_{x\in\overline\Omega} v(x) =\inf_{(x,y)\in\C} u(x,y).$$
\end{lemma}

\begin{proof} We argue by contradiction and we suppose that
$$ b_+ :=\min_{x\in\overline\Omega} v(x) > \inf_{(x,y)\in\C} u(x,y)
=b_-.$$
As a consequence, there exists
\begin{equation}\label{ell-def-0}
\ell\in \left(b_- ,\,b_+ \right).
\end{equation}
We define~$w(x,y):=\big(\ell-u(x,y)\big)^+$, and we observe that
$$ 0\le w \le \ell - \inf_{(x,y)\in\C} u(x,y) <
b_+ - b_- ,$$
so that $w\in L^\infty(\C)$. We claim that
\begin{equation}\label{ell-def-2}
\overline{\{ u<\ell \}} \,\subset\, {\overline\Omega}\times(0,+\infty) \quad \Longrightarrow \quad {\mbox{$w=0$ on $\partial_B \C$.}}
\end{equation}
To prove this, suppose by contradiction, that there exists
a sequence~$(x_k,y_k)\in\{ u<\ell \}$ with~$y_k\to0$ as~$k\to+\infty$.
Then, up to subsequence, we have that $x_k\to x_o$, for some~$x_o\in
\overline\Omega$, and by the continuity of $u$ (recall~\eqref{hp su u-2})
$$ \ell \ge \lim_{k\to+\infty }u(x_k,y_k)=u(x_o,0)=v(x_o)\ge
\min_{x\in\overline\Omega} v(x)=b_+.$$
This is in contradiction with~\eqref{ell-def-0}, and hence it
proves~\eqref{ell-def-2}. 

Now we take a smooth function~$\tau_R:[0,+\infty)\to[0,1]$ such that
$\tau_R=1$ in~$[0,R]$, $\tau_R=0$ in~$[2R,+\infty)$
and $|\tau'_R|\le 10/R$. We use~\eqref{WEAK}
with~$\varphi:= w\tau_R$ (notice that such function
lies in~${\mathcal{A}}$, thanks to~\eqref{hp su u-2}):
in this way, and recalling~\eqref{LO}
and~\eqref{ell-def-2},
we obtain that
\begin{equation}\label{WEAK:L2KJaaH78}
\begin{split}
0\,&=\int_{\C} a(y,|\nabla u|) \,\nabla u\cdot\nabla w\,\tau_R
+\int_{\C} a(y,|\nabla u|) \,\nabla u\cdot\nabla \tau_R\,w \\& = -
\int_{\C} a(y,|\nabla u|) \,|\nabla w|^2\,\tau_R
+\int_{\C} a(y,|\nabla u|) \,\nabla u\cdot\nabla \tau_R\,w
.\end{split}
\end{equation}
Now we use the positivity of $a$ and the boundedness of $w$ to see that
\begin{eqnarray*}
&& \left|\int_{\C} a(y,|\nabla u|) \,\nabla u\cdot\nabla \tau_R\,w\right|
\\&&\qquad\le \|w\|_{L^\infty(\C)}\, \sqrt{
\int_{\C} a(y,|\nabla u|) \,|\nabla u|^2 }\,
\sqrt{
\int_{\C} a(y,|\nabla u|) \,|\nabla \tau_R|^2 }\\
&&\qquad\le
\|w\|_{L^\infty(\C)}\, \sqrt{
\int_{\C} a(y,|\nabla u|) \,|\nabla u|^2 }\,
\sqrt{\frac{C}{R^2}
\int_{\Omega\times(R,2R)} a(y,|\nabla u|) },\end{eqnarray*}
and this quantity is infinitesimal as~$R\to+\infty$, thanks to~\eqref{NUOVE}. Using this and the Dominated Convergence Theorem,
we can pass to the limit into formula~\eqref{WEAK:L2KJaaH78}, obtaining
\begin{equation*}
0\le
-\int_{\C} a(y,|\nabla u|) \,|\nabla w|^2
.\end{equation*}
This implies that~$\nabla w$ vanishes identically, 
and so~$w$ is constant. Thus, recalling~\eqref{ell-def-2},
we conclude that~$w$ vanishes identically, and therefore~$u(x,y)\ge\ell$
for any~$(x,y)\in\C$. Accordingly, we have that $b_- \ge \ell$, in contradiction with~\eqref{ell-def-0}.
\end{proof}

\begin{corollary}\label{C:PT}
Under the previous notation, let
$$ c:= \min_{x\in\overline\Omega} v(x) = \inf_{(x,y) \in \mathcal{C}} u(x,y).$$
Then~$f(c)\le0$.
\end{corollary}

\begin{proof} Let $x_\star \in \overline{\Omega}$ be a minimum point for $v$, and let us assume by contradiction that~$f(c)>0$. By continuity (see~\eqref{NUOVE}),
there exists~$y_\star>0$ such that 
$$ a\big(y, |\nabla u(x_\star, y)|\big) \,\partial_y u(x_\star,y)
\le -\frac{f(c)}{2} \qquad {\mbox{for all }} y\in(0,y_\star].
$$
Since~$a(y,t)>0$ for any~$y>0$ and $t \ge 0$,
for every~$y\in (0,y_\star]$
$$ \partial_y u(x_\star,y)
\le -\frac{f(c)}{2}\,\frac{1}{
a\big(y, |\nabla u(x_\star, y)|\big) } .$$
Hence, for any~$y_o\in 
\left(0,\frac{y_\star}{2}\right]$
\begin{eqnarray*}
&& u(x_\star, y_\star) - u(x_\star,y_o)
= \int_{y_o}^{y_\star} \partial_y u(x_\star,y)\,dy \\ &&\qquad
\le -\frac{f(c)}{2} \int_{y_o}^{y_\star} \frac{dy}{
a\big(y, |\nabla u(x_\star, y)|\big)}
\le -\frac{f(c)}{2} \int_{y_\star/2}^{y_\star} \frac{dy}{
a\big(y, |\nabla u(x_\star, y)|\big)} =: -b_\star,
\end{eqnarray*}
and~$b_\star>0$. Taking the limit as~$y_o\to0$, we infer that
\[ \inf_{(x,y)\in\C} u(x,y) -\min_{x\in\overline\Omega} v(x)
\le u(x_\star, y_\star) - v(x_\star)
=u(x_\star, y_\star) - u(x_\star, 0) 
\le -b_\star,\]
in contradiction with Lemma \ref{0oPPy}.
\end{proof}

\subsection{Classification of stable solutions
in the case of convex/concave nonlinearities and end of the proof
of Theorem~\ref{TH:CC:1}}

With the previous preliminary work, we are now
in the position of completing the proof
of Theorem~\ref{TH:CC:1}. 
The proof will borrow an idea of~\cite{CH},
that is to test the equation against a vertical translation
of the solution (in this way, comparing the equation
with the linearized equation, the nonlinearity
is compared with its derivative, hence convexity
comes naturally into play).

\begin{proof}[Completion of the proof
of Theorem~\ref{TH:CC:1}]
We suppose that~$f$ is convex.

For any~$\varphi\in{\mathcal{A}}$,
we define
\begin{equation*}
J(\varphi):= -\int_{\C}
\frac{a_t(y,|\nabla u|)}{|\nabla u|}\,\big(\nabla u\cdot\nabla\varphi \big)^2.\end{equation*}
By~\eqref{DEF:B},
\begin{equation}\label{OIhg56}
\begin{split}
 \langle \mathcal{B}(y,\nabla u)
\nabla \varphi,\nabla \varphi\rangle
& = a(y,|\nabla u|)\,|\nabla\varphi|^2 +
\frac{a_t(y,|\nabla u|)}{|\nabla u|}\,\big(\nabla u\cdot\nabla\varphi \big)^2
\end{split}\end{equation}
Accordingly, and being $g \equiv 0$, the stability of $u$ implies that
\begin{equation}\label{stability:SEM}
\int_{\C} a(y,|\nabla u|) |\nabla \varphi|^2 -J(\varphi)
-\int_{\pa_B \C} f'(u) \varphi^2 
= I(\varphi)\ge 0,
\end{equation}
for any~$\varphi\in{\mathcal{A}}$. Let
$c:= \inf_{\mathcal{C}} u$, and let us consider a smooth function~$\tau_R:[0,+\infty)\to[0,1]$ such that
$\tau_R=1$ in~$[0,R]$, $\tau_R=0$ in~$[2R,+\infty)$
and $|\tau'_R|\le 10/R$.
We exploit first~\eqref{stability:SEM}
with~$\varphi(x,y):= \big(u(x,y)-c\big)\,\tau_R(y)$, and we obtain that
\begin{equation}\label{LKJH:OO1}
\begin{split}
0 \le 
\int_{\C} & a(y,|\nabla u|)\, |\nabla u|^2\,\tau_R^2 
+\int_{\C} a(y,|\nabla u|) \,|\nabla \tau_R|^2\,(u-c)^2
\\&+
2\int_{\C} a(y,|\nabla u|)\, (u-c)\,\tau_R \nabla u\cdot \nabla\tau_R
-J\big( (u-c)\,\tau_R\big) -
\int_{\pa_B \C} f'(u) (u-c)^2.
\end{split}
\end{equation}
Now we use\footnote{As a curiosity,
we stress that the choice of the test
function exploited to obtain~\eqref{LKJH:OO1}
is not the same as the one exploited to
obtain~\eqref{LKJH:OO2}.}~\eqref{WEAK} with~$\varphi:= 
\big(u(x,y)-c\big)\,\tau_R^2(y)$:
\begin{equation}\label{LKJH:OO2}
\begin{split}
&0=\int_{\C} a(y,|\nabla u|) \,|\nabla u |^2\,\tau_R^2
+ 2 \int_{\C} a(y,|\nabla u|) \,
(u-c)\,\tau_R \nabla u\cdot \nabla\tau_R
\\&\qquad\qquad
-\int_{\pa_B \C} f(u) \,(u-c).
\end{split}\end{equation}
Subtracting~\eqref{LKJH:OO2} from~\eqref{LKJH:OO1}, we obtain
\begin{equation}\label{LKJH:OO3}
\begin{split}
&0 \le
\int_{\C} a(y,|\nabla u|) \,|\nabla \tau_R|^2\,(u-c)^2
-J\big( (u-c)\,\tau_R\big)\\ 
&\qquad\qquad-
\int_{\pa_B \C} \big( f'(u) (c-u)+f(u)\big)\,(c-u).\end{split}
\end{equation}
Now we use the convexity of~$f$
to see that
\[
 f(u)+f'(u)\,(c-u)\le f(c).
\]
Since $f$ is strictly convex, the inequality is strict provided $\{u \neq c\} \neq \varnothing$.
Moreover, by Corollary~\ref{C:PT},
we know that~$u\ge c$, therefore
\begin{equation}\label{XF:g}
 \big(f(u)+f'(u)\,(c-u)\big)\,(c-u) \ge f(c)\,(c-u)
,\end{equation}
with strict inequality if~$\{ u\neq c\} \neq \varnothing$.
 
We also observe that, by \eqref{LO}, $-J((u-c)\tau_R) \le 0$. 
Plugging this and \eqref{XF:g} into \eqref{LKJH:OO3}, we conclude that
\begin{equation}\label{LKJH:OO3:BIS}
\begin{split}
&0 \le
\int_{\C} a(y,|\nabla u|) \,|\nabla \tau_R|^2\,(u-c)^2
-\int_{\pa_B \C} f(c)\,(c-u),\end{split}
\end{equation}
with strict inequality if ~$\pa_B \C \cap \{ u\neq c\} \neq \varnothing$. The previous inequality is satisfied for every $R>0$. Passing to the limit as $R \to +\infty$, we have  \begin{equation}\label{LKJH:OO4}
0 \le \lim_{R\to+\infty}
\int_{\C} a(y,|\nabla u|) \,|\nabla \tau_R|^2\,(u-c)^2 \le \lim_{R \to +\infty} \frac{C}{R^2} \int_{\Omega \times (R,2R)} a(y,|\nabla u|) 
=0,
\end{equation}
thanks to~\eqref{NUOVE}, the boundedness of $u$, and our choice of~$\tau_R$. Coming back to \eqref{LKJH:OO3:BIS}, this gives
\[
0 \le
\int_{\pa_B \C} f(c)\,(u-c)
\]
with strict inequality if~$\{ u\neq c\} \neq \varnothing$. 
But recalling that $f(c) \le 0$ and $u \ge c$ in $\mathcal{C}$, thanks to
Corollary~\ref{C:PT},
we have that the right hand side is nonpositive, so that the previous inequality cannot be strict; hence $\{u \neq c\} \cap \pa_B \C= \varnothing$, i.e. $u$ is constant on $\pa_B \C$.

To complete the proof, we come back to~\eqref{LKJH:OO2}. Using the fact that $u =c$ on $\pa_B \C$, the last integral there vanishes, so that
\begin{equation}\label{conclusion}
\begin{split}
0 & = \int_{\C} a(y,|\nabla u|) \,|\nabla u |^2\,\tau_R^2
+ 2 \int_{\C} a(y,|\nabla u|) \,
(u-c)\,\tau_R \nabla u\cdot \nabla\tau_R
 \\
\end{split}
\end{equation}
for every $R>0$. Moreover
\begin{multline*}
\left|\int_{\C} a(y,|\nabla u|)  (u-c) \tau_R \nabla u \cdot \nabla \tau_R\right| = \left|\int_{\Omega \times (R,2R)} a(y,|\nabla u|)  (u-c) \tau_R \nabla u \cdot \nabla \tau_R \right|\\
 \le \frac12 \int_{\Omega \times (R,2R)} a(y,|\nabla u|) |\nabla u|^2 + \frac{1}{2}\int_{\Omega \times (R,2R)} a(y,|\nabla u|)(u-c)^2 |\nabla \tau_R|^2 \to 0
\end{multline*}
as $R \to +\infty$, where we used the first integrability assumption in~\eqref{NUOVE} and \eqref{LKJH:OO4}. Therefore, passing to the limit as $R \to +\infty$ into \eqref{conclusion}, we deduce by the Monotone Convergence Theorem that
\[
 \int_{\C} a(y,|\nabla u|) \,|\nabla u |^2= 0,
\]
i.e. $u$ is constant in the whole $\C$.
\end{proof}

\section{Application to nonlocal problems with Neumann boundary conditions}\label{sec: s-Neumann}

Before proceeding with the proofs of Theorems \ref{thm: s-Neumann 1} and \ref{thm: s-Neumann 2}, we observe that, if $v$ is a solution to \eqref{s-Neumann}, then its extension $u$ satisfies the integrability condition \eqref{integrabilty assumption}.

\begin{lemma}\label{lem: agg}
Let~$v\in C^2(\overline{\Omega})$
be a solution of~\eqref{s-Neumann}.
Let~$u$ be the extension of~$v$.
Then, $u$
satisfies~\eqref{integrabilty assumption}
with~$a\equiv1$ and~$g\equiv0$.
\end{lemma}

\begin{proof}
First of all, we show that, for any~$\beta
\in\left(0,\frac{2}{n}\right)$
there exists~$K(\beta) \in\N$ such
that, for any~$
k\in\N$ with~$k\ge K(\beta)$, we have
\begin{equation}\label{l k}
\lambda_k > k^\beta,
\end{equation}
where the~$\lambda_k$'s are the 
eigenvalues of~$-\Delta$ in~$\Omega$
with homogeneous Neumann condition.
To prove this, we argue by contradiction
and assume that there exists a sequence~$k_j$
such that~$\lambda_{k_j} \le k_j^\beta$.
We denote by~$N(\Lambda)$ the
number of eigenvalues 
which are strictly smaller than~$\Lambda^2$.
Then, since
$$ \lambda_1\le \dots\le \lambda_{k_j}\le
k_j^\beta,$$
we have that
\begin{equation}\label{7uhsGA}
N(k_j^{\frac\beta2}) \ge k_j.
\end{equation}
On the other hand, by the Weyl
Asymptotic Formula (see e.g.~\cite{WE}),
we have that
$$ N(\Lambda) = C_\star \Lambda^n +
o(\Lambda^n) \le 2C_\star \Lambda^n,$$
as~$\Lambda\to+\infty$,
for some~$C_\star>0$. So, taking~$\Lambda:=
k_j^{\frac\beta2}$, we have that
$$ N(k_j^{\frac\beta2}) \le 2C_\star
k_j^{\frac{\beta n}{2}},$$
as~$j\to+\infty$. By comparing this with~\eqref{7uhsGA},
we obtain that~$k_j\le 2C_\star
k_j^{\frac{\beta n}{2}}$, as~$j\to+\infty$.
That is,
$$ 1\le 2C_\star \lim_{j\to+\infty}
k_j^{\frac{\beta n}{2}-1} =0.$$
This is a contradiction, and so~\eqref{l k}
is proved.

Now we define~$h(x):=f(v(x))$.
We also set
$$ h_k :=\int_\Omega 
h(x)\,\varphi_k(x)\,dx.$$
We remark that~$\partial_\nu h=
f'(v)\,\partial_\nu v=0$ along~$
\partial\Omega$. Therefore
\begin{eqnarray*}
-\lambda_k h_k &=&
-\lambda_k\int_\Omega 
h\,\varphi_k = \int_\Omega
h \Delta\varphi_k
\\ &=&
\int_\Omega
\Delta h\,\varphi_k
-{\rm div}\,\big(\varphi_k \nabla h\big)
+{\rm div}\,\big(h \nabla \varphi_k \big)
\,dx = \int_\Omega
\Delta h\,\varphi_k.
\end{eqnarray*}
Since~$f$ is~$C^2$, then so is~$h$,
and thus we find that
\begin{equation}\label{IO09}
\lambda_k\,|h_k|\le
\| h\|_{C^2(\Omega)}\,
\int_\Omega |\varphi_k| \le C,
\end{equation}
for some~$C>0$.

Now we remark that~$u$ can be written
as
$$ u(x,y) =\sum_{k=0}^{+\infty}
h_k\,\varphi_k(x)\, e^{-\sqrt{\lambda_k} y}.$$
So, if we set
$$ \tilde u(x,y) := u(x,y)- h_0\varphi_0(x)=
u(x,y)- \frac{h_0}{\sqrt{|\Omega|}},$$
we obtain that, for any~$y\ge2$,
\begin{eqnarray*}
|\tilde u(x,y)|&\le&
\sum_{k=1}^{+\infty}
|h_k|\,|\varphi_k(x)|\, 
e^{-\sqrt{\lambda_k} y} \\
&\le& e^{-\frac{\sqrt{\lambda_1} y}{2}}
\sum_{k=1}^{+\infty}
|h_k|\,|\varphi_k(x)|\,
e^{-\frac{\sqrt{\lambda_k} y}{2}}\\
&\le& 
e^{-\frac{\sqrt{\lambda_1} y}{2}}
\sum_{k=1}^{+\infty} \frac{C}{\lambda_k}
\,|\varphi_k(x)|\,
e^{-\sqrt{\lambda_k} },
\end{eqnarray*}
thanks to~\eqref{IO09}.
Now we claim that
\begin{equation}\label{PHIk}
\|\varphi_k\|_{L^\infty(\Omega)}\le C_1\,\lambda_k^{C_2},
\end{equation}
for some $C_1$, $C_2>0$.
To prove this, we use a Moser iteration method.
Namely, we observe that, for $\delta\ge 0$, testing
the eigenvalue equation against $\varphi_k^{1+\delta}$,
we have
$$ \frac{1+\delta}{\left( 1+\frac{\delta}{2}\right)^2 }
\int_\Omega \left|\nabla \varphi_k^{1+\frac\delta2}\right|^2 =
\int_\Omega \nabla\varphi_k \cdot\nabla \varphi_k^{1+\delta}= 
-\int_\Omega \Delta\varphi_k \varphi_k^{1+\delta} 
= \lambda_k \int_\Omega \varphi_k^{2+\delta} .$$
Therefore, by the Sobolev embedding,
$$ \frac{1+\delta}{\left( 1+\frac{\delta}{2}\right)^2 }\,S\,
\left[ \int_\Omega \varphi_k^{2_* \, \left(1+\frac\delta2\right)}
\right]^{\frac2{2_*}}
\le \lambda_k \int_\Omega \varphi_k^{2+\delta} ,$$
with $2_* > 2$ and $S>0$.

Therefore, we can choose the recursive sequence $\delta_0:=0$
and $\delta_{j+1}:=2_* -2 +\frac{2_*}{2}\delta_j$, and we 
remark that $(2+\delta_{j+1}) \frac{2}{2_*}=2+\delta_j$,
so we 
obtain
$$ \| \varphi_k\|_{ L^{2+\delta_{j+1}}(\Omega) }
\le 
\left[ \frac{\left( 1+\frac{\delta_j}{2}\right)^2 }{1+\delta_j}
\,\frac{\lambda_k}{S}\right]^{\frac{1}{2+\delta_j}}\, 
\| \varphi_k\|_{ L^{2+\delta_{j}}(\Omega) }. $$
This gives that
\begin{eqnarray*}
\| \varphi_k\|_{ L^{2+\delta_{j+1}}(\Omega) } &\le&
\prod_{i=0}^{j}
\left[ 
\frac{\left( 1+\frac{\delta_i}{2}\right)^2 }{1+\delta_i}
\,\frac{\lambda_k}{S}\right]^{\frac{1}{2+\delta_i}} 
\\ &=&
\exp \left[
\sum_{i=0}^{j} 
\frac{1}{2+\delta_i}
\log
\left[
\frac{\left( 1+\frac{\delta_i}{2}\right)^2 }{1+\delta_i}
\,\frac{\lambda_k}{S}\right] 
\right]
\\&\le&
\exp \left[
\sum_{i=0}^{+\infty}
\frac{\left|
\log \frac{\left( 1+\frac{\delta_i}{2}\right)^2 }{1+\delta_i}
\right|}{2+\delta_i}
+\frac{
\log\frac{\lambda_k}{S}
}{2+\delta_i}
\right] .
\end{eqnarray*}
Since $\delta_i\ge \left(\frac{2_*}{2}\right)^{i-1}\,(2_*-2)$
for any $i\in\N$, we conclude that
$$ \| \varphi_k\|_{ L^{2+\delta_{j+1}}(\Omega) } \le
\exp \left[ C\left(1+
\log\lambda_k \right)
\right] \le C_1 \lambda_k^{C_2},$$
for some $C$, $C_1$, $C_2>0$.
Then, we send $j\to+\infty$ and we obtain \eqref{PHIk}.

Therefore, up to renaming the constants, we conclude that,
for any~$y\ge2$,
$$|\tilde u(x,y)|\le
C e^{-\frac{\sqrt{\lambda_1} y}{2}}
\sum_{k=1}^{+\infty} \lambda_k^{C-1} 
e^{-\sqrt{\lambda_k}} 
\le Ce^{-\frac{\sqrt{\lambda_1} y}{2}},$$
where the latter inequality is a consequence
of~\eqref{l k} (and of the fact that
the exponential decays faster than
any power). Since~$\tilde u$ is harmonic,
by elliptic estimates
we find that also the derivatives of~$\tilde u$
(and so of~$u$) are bounded by~$C
e^{-\frac{ \sqrt{\lambda_1} y}{2}}$,
for any~$y\ge3$. {F}rom this,
one obtains~\eqref{integrabilty assumption}.
\end{proof}

\begin{proof}[Proof of Theorems \ref{thm: s-Neumann 1} and \ref{thm: s-Neumann 2}] 
Let $v \in H^{1/2}(\Omega) \cap L^\infty(\Omega)$ be a solution to \eqref{s-Neumann}. Then, by \cite{StinVolz}, $v$ is the trace on $\Omega \times\{0\}$ of a weak solution $u \in \mathcal{H}(\C)$ to \eqref{s-Neumann extended}:
\[
\begin{cases}  
\Delta u =0 & \text{in $\C$}\\
\pa_\nu u=0 & \text{on $\pa_L \C$} \\
-\pa_y u = f(u) & \text{on $\pa_B \C$}.
\end{cases}
\]
Moreover, arguing as in Theorem 3.5-part~4 in~\cite{StinVolz}
(recalling that~$f\in C^{2,\alpha}(\R)$ and~$\Omega$
is of class~$C^{4,\alpha}$, and using
higher order Schauder estimates), one sees
that $u \in C^{3,\alpha}(\overline{\C})$. Thus, Lemma \ref{lem: agg} implies that \eqref{integrabilty assumption} is satisfied, and hence~$u$
is a classical, stable solution of \eqref{s-Neumann extended}, 
and satisfies all the assumptions in~\eqref{hp su u-2}.

As a consequence, if $\Omega$ is convex we are in position to apply Theorem \ref{P-CA}, deducing that $u$ depends only on $y$. Therefore, $v=u(\cdot,0)$ has to be constant.

On the other hand,
if $f$ is either strictly convex, or strictly concave, Theorem \ref{TH:CC:1} implies in the same way that $v$ is constant.
\end{proof}

\section{A counterexample}\label{HJ:AO}

Now we prove the statement given in Example~\ref{EXAMPLE}

\begin{proof}[Proof of Example~\ref{EXAMPLE}]
We let~$\epsilon\in(0,1)$ and~$h_*\in C^2(\R)$ be such that~$h_*(x)=h(x)$
for any~$x\in[-1,1]$, $h_*(x)=2x$
if~$|x|\in \left[1+\frac{\epsilon}{11},1
+\frac{2\epsilon}{11}\right]$, and~$h_*(x)=-2x$
if~$|x|\in \left[1+\frac{3\epsilon}{11},1
+\frac{4\epsilon}{11}\right]$.

Then, by Theorem 1.1 in~\cite{ALL}, there exists~$v\in
C^2( (-2,2))\cap C(\R)$ which is
compactly supported, such that~$(-\Delta)^s v=0$
in~$(-2,2)$ and~$\|v-h_*\|_{C^2 ( (-2,2))}
\le\epsilon$. In particular,
$$ \|v-h\|_{C^2 ((-1,1))}
=\|v-h_*\|_{C^2 ((-1,1))}
\le\epsilon.$$
Moreover, if~$x\in\left[1+\frac{\epsilon}{11},1
+\frac{2\epsilon}{11}\right]$, we have that
$$ v'(x) \ge h_*'(x) - \|v-h_*\|_{C^1 ((-2,2))}
\ge2 - \|v-h_*\|_{C^2 ((-2,2))} \ge 1.$$
Similarly, if~$x\in\left[1+\frac{3\epsilon}{11},1
+\frac{4\epsilon}{11}\right]$, we have that~$v'(x)\le-1$.

As a consequence, there exist~$\delta_1$, $\delta_2\in 
\left[\frac{\epsilon}{11},\frac{4\epsilon}{11}\right]$ such that~$v'(-1-\delta_1)=
v'(1+\delta_2)=0$.
In particular, for any~$x\in\{-1-\delta_1,\,1+\delta_2\}$,
$$ \lim_{\substack{y\in (-1-\delta_1,1+\delta_2) \\ y\to x}}
\frac{v(x)-v(y)}{|x-y|^s}=0.$$
Then, $v$ satisfies the desired result.
\end{proof}

\section*{Acknowledgements}
Part of this work was carried out while Serena Dipierro
and Nicola Soave were visiting the {\it
Weierstra{\ss}-Institut f\"ur Angewandte
Analysis und Stochastik} in Berlin, which
they wish to thank for the hospitality. 

This work has been supported by the Humboldt Foundation, the
ERC grant 277749 {\it E.P.S.I.L.O.N.} ``Elliptic
Pde's and Symmetry of Interfaces and Layers for Odd Nonlinearities'',
the PRIN grant 201274FYK7
``Aspetti variazionali e
perturbativi nei problemi differenziali nonlineari''
and
the ERC grant 339958 {\it Com.Pat.} ``Complex Patterns for 
Strongly Interacting Dynamical Systems''.


\begin{thebibliography}{99}

\bibitem{NIR} {\sc S. Agmon, A. Douglis and L. Nirenberg},
{\em Estimates near the boundary for solutions of elliptic partial differential equations satisfying general boundary conditions. I}.
Commun. Pure Appl. Math. 12, 623--727 (1959).

\bibitem{NIR2} {\sc S. Agmon, A. Douglis and L. Nirenberg},
{\em Estimates near the boundary for solutions of elliptic partial differential equations satisfying general boundary conditions. II.}
Commun. Pure Appl. Math. 17, 35--92 (1964).

\bibitem{Cabre-Cinti-1} {\sc X. Cabr\'e and E. Cinti},
{\em Energy estimates and $1$-D symmetry for
nonlinear equations involving the half-Laplacian}.
Discrete Contin. Dyn. Syst. 28, no. 3, 1179--1206 (2010).

\bibitem{Cabre-Cinti-2} {\sc X. Cabr\'e and E. Cinti},
{\em Sharp energy estimates for nonlinear fractional diffusion equations}.
Calc. Var. Partial Differ. Equ. 49, no. 1-2, 233--269 (2014).

\bibitem{Cabre-Sire}
{\sc X. Cabr\'e and Y. Sire},
{\em  Nonlinear equations for fractional Laplacians. II: Existence, uniqueness, and qualitative properties of solutions}.
Trans. Amer. Math. Soc. 367, no. 2, 911--941 (2015).

\bibitem{Sola}
{\sc X. Cabr\'e and J. Sol\`a-Morales},
{\em Layer solutions in a half-space for boundary reactions}.
Commun. Pure Appl. Math. 58, no. 12, 1678--1732 (2005).

\bibitem{CH} {\sc R. G. Casten and C. J. Holland},
{\em Instability results for reaction-diffusion equations with Neumann boundary conditions}.
J. Differ. Equations 27, 266--273 (1978).

\bibitem{DPV-CMP}
{\sc S. Dipierro, G. Palatucci and E. Valdinoci},
{\em Dislocation dynamics in crystals: a macroscopic theory in a fractional Laplace setting}.
Commun. Math. Phys. 333, no. 2, 1061--1105 (2015).

\bibitem{ALL} {\sc S. Dipierro, O. Savin and E. Valdinoci},
{\em All functions are locally $s$-harmonic up to a small error}.
To appear in J. Eur. Math. Soc. (JEMS)
{\tt http://arxiv.org/abs/1404.3652}

\bibitem{FALL} {\sc M. M. Fall and S. Jarohs},
{\em Overdetermined problems with fractional Laplacian}.
ESAIM Control Optim. Calc. Var. 21, no. 4, 924--938 (2015).

\bibitem{Farina} {\sc A. Farina}, 
{\em Propri\'et\'es qualitatives de solutions d’\'equations et syst\`emes
d’\'equations nonlin\'eaires}.
Habilitation \`a diriger des recherches, Paris (2002). 

\bibitem{FaScVa} {\sc A. Farina, B. Sciunzi and E. Valdinoci}, 
{\em Bernstein and De Giorgi type problems: new results via a geometric approach}. 
Ann. Sc. Norm. Super. Pisa, Cl. Sci. (5) 7, no. 4, 741--791 (2008).

\bibitem{Fisher}
{\sc R. A. Fisher},
{\em The genetical theory of natural selection}.
London: Oxford University Press (1930).

\bibitem{GT-rep-1998}
{\sc D. Gilbarg and N. S. Trudinger},
{\em Elliptic partial differential equations of second order}.
Reprint of the 1998 ed.
Classics in Mathematics. Berlin: Springer (2001).

\bibitem{Grubb}
{\sc G. Grubb},
{\em Local and nonlocal boundary conditions for
              {$\mu$}-transmission and fractional elliptic
              pseudodifferential operators}.
Anal. PDE 7, no. 7, 1649--1682, 2014.

\bibitem{HAN-LIN} {\sc Q. Han and F. Lin},
{\em Elliptic partial differential equations}.
Courant Lecture Notes in Mathematics.
New York: New York University (1997).

\bibitem{Hirth-Lothe}
{\sc J. P. Hirth and J. Lothe},
{\em Theory of Dislocations}.
Virginia: Krieger Publishing Co. (1982).

\bibitem{Kolmogorov-Moscow-1937}
{\sc A. Kolmogoroff, I. Pretrovsky and N. Piscounoff},
{\em \'Etude de l'\'equation de la diffusion avec
croissance de la quantite de mati\`ere et 
son application \`a un probl\`eme biologique}.
Bull. Univ. \'Etat Moscou, S\'er. Int., Sect. A: 
Math. et M\'ecan. 1, no. 6, 1--25 (1937).

\bibitem{MonPelVer}
{\sc E. Montefusco, B. Pellacci, G. Verzini},
{\em Fractional diffusion with {N}eumann boundary conditions: the logistic equation}. Discrete Contin. Dyn. Syst. Ser. B 18, no. 8, 2175--2202 (2013).


\bibitem{WE}
{\sc Y. Netrusov, Y. Safarov},
{\em {W}eyl asymptotic formula for the {L}aplacian on domains with rough boundaries}. Comm. Math. Phys. 253, no. 2, 481--509 (2005).

\bibitem{ROS} {\sc X. Ros-Oton and J. Serra},
{\em The Pohozaev identity for the fractional Laplacian}.
Arch. Ration. Mech. Anal. 213, no. 2, 587--628 (2014).

\bibitem{ROS-2} {\sc X. Ros-Oton and J. Serra},
{\em The Dirichlet problem for the fractional Laplacian: regularity up to the boundary}.
J. Math. Pures Appl. (9) 101, no. 3, 275--302 (2014).

\bibitem{savin} {\sc O. Savin and E. Valdinoci},
{\em Some monotonicity results for minimizers in the calculus of variations}.
J. Funct. Anal. 264, no. 10, 2469--2496 (2013).

\bibitem{different} {\sc R. Servadei and E. Valdinoci},
{\em On the spectrum of two different fractional operators}.
Proc. Roy. Soc. Edinburgh Sect. A 144, no. 4, 831--855 (2014).

\bibitem{SirVal1} {\sc Y. Sire and E. Valdinoci}, 
{\em Fractional Laplacian phase transitions and boundary reactions: 
a geometric inequality and a symmetry result}. 
J. Funct. Anal. 256, no. 6, 1842--1864 (2009).

\bibitem{SirVal2} {\sc Y. Sire and E. Valdinoci}, 
{\em Rigidity results for some boundary quasilinear phase transitions}. 
Commun. Partial Differ. Equations 34, no. 7, 765--784 (2009).

\bibitem{Nicola} {\sc N. Soave and E. Valdinoci},
{\em Overdetermined problems for the fractional Laplacian in exterior and annular sets}.
To appear on J. Anal. Math. Preprint {\tt http://arxiv.org/abs/1412.5074}
(2014).

\bibitem{Stern} {\sc P. Sternberg and K. Zumbrun},
{\em A Poincar\'e inequality with applications to volume-constrained area-minimizing surfaces}.
J. Reine Angew. Math. 503, 63--85 (1998).

\bibitem{Stinga} {\sc P. R. Stinga and J. Torrea},
{\em Extension problem and Harnack's inequality for some fractional operators}.
Comm. Partial Differential Equations 35, no. 11, 2092–2122 (2010).

\bibitem{StinVolz}{\sc P. R. Stinga and B. Volzone},
{\em Fractional semilinear Neumann problems arising from a fractional Keller-Segel model}.
Calc. Var. Partial Differential Equations 54, no. 1, 1009-1042 (2015).

\end{thebibliography}
\end{document}